\documentclass[reqno]{amsart}

\usepackage{amsmath}
\usepackage{amsfonts}
\usepackage{amsthm}
\usepackage{amssymb}
\usepackage{graphicx}
\usepackage{graphics}
\usepackage{bm}
\usepackage{dsfont}
\usepackage{color}
\usepackage{enumerate}
\usepackage[font=footnotesize]{caption}
\usepackage[all]{xy}
\usepackage{xypic}
\allowdisplaybreaks

\numberwithin{equation}{section}

\newtheorem{thm}{Theorem}[section]

\newtheorem{lem}[thm]{Lemma}
\newtheorem{prop}[thm]{Proposition}
\theoremstyle{definition}

\newtheorem{rmk}[thm]{Remark}

\newcommand{\N}{\mathds{N}}

\newcommand{\Z}{\mathds{Z}}

\newcommand{\R}{\mathds{R}}

\newcommand{\E}{\mathbb{E}}

\newcommand{\C}{\mathds{C}}
\newcommand{\T}{\mathds{T}}

\newcommand{\diff}{\mathrm{d}}

\newcommand{\M}{\mathcal{M}}
\newcommand{\U}{\mathcal{U}}

\newcommand{\V}{\mathcal{V}}

\newcommand{\A}{\mathbb{A}}

\newcommand{\q}{{\mathbf{q}}}
\newcommand{\p}{\mathbf{p}}
\newcommand{\vv}{\mathbf{v}}

\newcommand{\x}{\mathbf{x}}
\newcommand{\w}{\mathbf{w}}
\newcommand{\HH}{\mathbb H}

\newcommand{\dt}{\frac{\diff }{\diff t}}
\newcommand{\ds}{\frac{\diff }{\diff s}}

\newcommand{\emb}{\mathrm{emb}}

\newcommand{\crit}{\mathrm{crit}}
\newcommand{\parall}{\mathrm{par}}
\newcommand{\ver}{{\mathrm{v}}}
\newcommand{\hor}{{\mathrm{h}}}

\newcommand{\hhh}{\mathbf{h}}

\newcommand{\std}{{\mathrm{std}}}

\begin{document}

\title[The Palais-Smale condition for the Hamiltonian action and applications]{The Palais-Smale condition for the Hamiltonian action on a mixed regularity space of loops in cotangent bundles and applications}


\author[L. Asselle]{Luca Asselle}
\address{Justus Liebig Universit\"at Giessen, Mathematisches Institut, Arndtstrasse 2, 35392 Giessen, Germany}
\email{luca.asselle@math.uni-giessen.de}

\author[M. Starostka]{Maciej Starostka}
\address{Gda\'nks Institute of Technology, Gabriela Narutowicza 11/12, 80233 Gda\'nsk, Poland}
\email{maciej.starostka@ruhr-uni-bochum.de}

\date{March, 2020}

\begin{abstract}
We  show that the Hamiltonian action satisfies the Palais-Smale condition over a ``mixed regularity'' space of loops in cotangent bundles,
namely the space of loops with regularity $H^s$, $s\in (\frac 12, 1)$, in the base and $H^{1-s}$ in the fiber direction. As an application, we give a simplified proof 
of a theorem of Hofer-Viterbo on the existence of closed characteristic leaves for certain contact type hypersufaces in cotangent bundles.
\end{abstract}

\maketitle

\vspace{-3mm}

\section{Introduction}

Let $(W,\omega)$ be a closed symplectic manifold, and let $H:\T\times W \to \R$ be a smooth time-depending Hamiltonian, where $\T:=\R/\Z$. With the pair $(H,\omega)$ we can associate 
an Hamiltonian vector field $X_H$ by
$$\imath_{X_H}\omega (\cdot) = -\diff H(\cdot),$$
and hence an induced Hamiltonian system by
\begin{equation}
\dot x = X_H(x).
\label{hamsystem}
\end{equation}
One of the central problem in the theory of Hamiltonian systems is to find (one-)periodic solutions of \eqref{hamsystem}. Such periodic solutions can be found as 
critical points of a suitable action functional: the \textit{Hamiltonian action} of a contractible loop $x:\T\to W$ is given by 
\begin{equation}
\A_H(x) := \int_{\mathbb D} \bar x^*\omega - \int_\T H(t,x(t))\, \diff t,
\label{hamactiongeneral}
\end{equation}
where $\bar x:\mathbb D\to W$ is a map on the disk $\mathbb D$ coinciding with $x$ on $\partial \mathbb D\cong \T$. For an arbitrary $(W,\omega)$, the functional $\A_H$ 
is unfortunately not well-suited for finding critical points using classical Morse theory, and this has forced to develop new techniques to deal with the functional $\A_H$. One of the most powerful 
is certainly \textit{Floer theory}: The Floer homology $FH_*(W,\omega)$ of $(W,\omega)$ is the 
homology of a chain complex which is generated by contractible one-periodic solutions of \eqref{hamsystem}. The boundary operator is defined by a suitable count of ``negative $L^2$-gradient flow lines'' 
of $\A_H$; these are cylinders $u:\R\times \T\to W$ which are asymptotic to pairs of periodic orbits of $X_H$ and solve the nonlinear perturbed Cauchy-Riemann equation 
\begin{equation}
\partial_s u + J_t(u) (\partial_t u - X_H(t,u))=0,
\label{cauchyriemann}
\end{equation}
where $(J_t)$ is a given loop of $\omega$-compatible almost complex structures on $W$. As the notation suggests, $FH_*(W,\omega)$ does not depend on the defining data $H$ and $J$, and it is 
actually isomorphic to the singular homology of $M$ with respect to suitable coefficient rings. This approach to the study of periodic orbits on general symplectic manifolds was introduced by Floer in 
the late 80's \cite{Floer:1988oq,Floer:1989xr,Floer:1989ud} under additional assumptions, and later extended more and more by several authors, see e.g. \cite{HS:95,LT:98,FO:99}. 
Floer homology can be defined also for non-compact symplectic manifolds which are suitably convex at infinity. In this case, the theory requires the use of Hamiltonians having a suitable behavior at infinity 
and is a genuine infinite dimensional homology theory: for instance, the Floer homology of $T^*M$, the total space of the cotangent bundle of a closed manifold $M$, is isomorphic to the singular 
homology of the free loop space of $M$, see \cite{Abbondandolo:2006jf,AS:14,Salamon:2006mh}.

On particular symplectic manifolds however, a Morse theory for the Hamiltonian action functional $\A_H$ can be obtained by more classical methods. This is the case of the torus $\T^{2n}$, for which 
$\A_H$ admits a smooth negative gradient flow on the space of contractible loops of Sobolev class $H^{1/2}$. The space of loops of class $H^{1/2}$ in an arbitrary manifold does not have a good structure of an 
infinite dimensional manifold due to the fact that curves of class $H^{1/2}$ might have discontinuities, but since $\T^{2n}$ is a quotient of $\R^{2n}$, the space of contractible $H^{1/2}$-loops on $\T^{2n}$ 
can be identified with $\T^{2n}$ times the Hilbert space of $H^{1/2}$-loops in $\R^{2n}$ having zero mean. Although strongly indefinite (meaning that all its critical points have infinite Morse index and co-index),
the functional $\A_H$ has good analytical properties on this space. By using finite dimensional approximations, the $H^{1/2}$-approach was used by Conley and Zehnder \cite{Conley:1984xb} to prove Arnold's conjecture on 
$\T^{2n}$ five years before the birth of Floer homology; see also \cite{Starostka:2019} for a simplified proof. Another symplectic manifold which can be dealt with by similar methods is $\C\mathbb P^n$, see \cite{For:85}.

\vspace{2mm}

In this and a follow up paper we aim at enlarging the class of symplectic manifolds such that the action functional $\A_H$ given by \eqref{hamactiongeneral} induces a negative gradient flow with good compactness properties on a suitable space of free loops. In the present paper we will focus on the class of symplectic manifolds given by cotangent bundles $T^*M$ over a closed manifold $M$: $T^*M$ carries a natural symplectic form $\omega_\std$, which in local coordinates $(q,p)=(q_1,p_1,...,q_n,p_n)$ is given by $\omega_\std = \diff q \wedge \diff p$. 
In this setting, the functional $\A_H$ reads
$$\A_H(x) = \int_\T x^*\lambda_\std - \int_\T H(t,x(t))\, \diff t,$$
where $\lambda_\std = p\diff q$ is the \textit{Liouville one-form}. As domain of definition of $\A_H$ we will take the bundle $\M^{1-s}$ over the Hilbert manifold of loops $H^s(\T,M)$, $s\in (\frac12,1)$ whose typical fibre is given by the space of $H^{1-s}$-vector fields along $\gamma\in C^\infty(\T,M)$; for more details we refer to Section \ref{section:2}.

Other classes of manifolds that we aim at studying are given by twisted cotangent bundles and toric manifolds respectively. In the latter case, 
the isotropic foliation given by the torus action will play the role of the fibers of $T^*M$. We will address these question in a forthcoming paper. 

We recall that a $C^1$-functional $f:\mathcal H\to \R$, $\mathcal H$ Hilbert manifold, satisfies the Palais-Smale condition if every sequence $(\gamma_n)\subset \mathcal H$ 
such that 
$$f(\gamma_n) \to c, \quad \diff f (\gamma_n) \to 0,$$
admits a converging subsequence. 

\begin{thm}
Let $M$ be a closed manifold, and let $\pi:T^*M\to M$ be its cotangent bundle. Furthermore, let $H:\T\times T^*M\to \R$ be a smooth time-depending Hamiltonian function satisfying the growth condition 
$$H(t,q,p) = \frac 12 |p|_q^2 + c, \quad \forall (q,p)\in T^*M\setminus K, \ \forall t\in \T,$$
where $K\subset T^*M$ is a compact subset, $|\cdot|$ is the norm induced by a Riemannian metric on $M$ and $c\in \R$ is some constant. Then, for every $s\in (\frac 12 ,1)$,
 $\A_H:\mathcal M^{1-s}\to \R$ satisfies the Palais-Smale condition.
\label{thm:1}
\end{thm}

The Palais-Smale condition is, as the natural replacement of compactness, a key property in infinite-dimensional critical point theory, and, as such, it is the starting point 
to obtain a ``classical'' Morse theory for the Hamiltonian action functional $\A_H$. Indeed, once one has a negative gradient
flow with good analytical properties for a strongly indefinite functional, one can obtain a Morse theory e.g. using the Morse complex approach which is developed in 
\cite{AM:01,AM:03,AM:05,AM:09}. In this approach, one constructs a chain complex looking at one-dimensional intersections 
of unstable and stable manifolds of pairs of critical points. The difference with respect to Floer homology is that the Cauchy-Riemann equation \eqref{cauchyriemann} is replaced by an ODE in an 
infinite dimensional manifold. We will address this problem in a forthcoming paper. 

In this paper, we will apply Theorem \ref{thm:1} to give a simplified proof of a Theorem of Hofer and Viterbo \cite{Hofer:1988} on the existence of closed characteristic leaves for certain contact type hypersurfaces in $T^*M$. To this purpose, we recall that solutions of \eqref{hamsystem} for an autonomous (that is, time independent) Hamiltonian function 
$H:T^*M\to \R$ are contained in a level set of $H$; indeed, for any solution $x:I\to T^*M$ of \eqref{hamsystem} we have 
$$\frac{\diff}{\diff t} H\circ x(t) = \diff H(x(t))[\dot x(t)] = -\omega_\std (X_H(x(t)),\dot x(t)) = -\omega_\std(\dot x(t),\dot x(t))=0.$$

We set $\Sigma:=H^{-1}(\kappa)$, $\kappa \in\R$, and suppose that $\Sigma$ is compact, connected, 
and \textit{regular}, that is, $X_H$ is nowhere vanishing on $\Sigma$. As it is well-known, the 
Hamiltonian dynamics on $\Sigma$ essentially depends only on $\Sigma$, meaning that the dynamics of two different Hamiltonians both defining 
$\Sigma$ only differ by time-reparametrization: The symplectic form $\omega_\std$ induces a line distribution on $\Sigma$ via 
$$\ell_\Sigma := \ker \, \omega_\std |_{T^*\Sigma},$$
and $X_H |_\Sigma \in \ell_\Sigma$.  The line distribution $\ell_\Sigma \to \Sigma$ is usually called the \textit{characteristic line bundle} over $\Sigma$ and induces a foliation 
of $\Sigma$ (whose leaves are unparametrized Hamiltonian trajectories), called the \textit{characteristic foliation} of $\Sigma$. In particular, finding periodic solutions to \eqref{hamsystem} with energy $\kappa$ is equivalent to finding closed characteristic leaves on $\Sigma=H^{-1}(\kappa)$. 
In what follows we say that an hypersurface $\Sigma\subset T^*M$ is $\mathbb O_M$-\textit{separating} if the bounded component of $T^*M\setminus \Sigma$ contains the zero-section $\mathbb O_M$ of the bundle $T^*M\to M$.

\begin{thm}
Let $\Sigma \subset T^*M$ be a compact connected $\mathbb O_M$-separating contact type hypersurface. Then there exists a closed characteristic leaf on $\Sigma$.
\label{thm:2}
\end{thm}

The hypersurface $\Sigma\subset (T^*M,\omega_\std)$ is called of \textit{contact type}, if there exists a one-form $\alpha\in \Omega^1(\Sigma)$ such that $\omega_\std |_\Sigma= \diff \alpha$ 
and $\alpha$ does not vanish on $\ell_\Sigma$, or, equivalently, if there exists a Liouville vector field $Y$ on a neighborhood $U$ of $\Sigma$ 
(meaning that $L_Y \omega_\std=\omega_\std$ on $U$, where $L$ denotes the Lie derivative) which is everywhere transverse to $\Sigma$ (c.f. \cite[Section 4.3]{Hofer:1994bq}).
In contact geometry, one of the most famous open conjecture - universally known as the \textit{Weinstein conjecture} -
states that every closed contact manifold possesses a closed Reeb orbit (in our language, a closed 
charateristic leave). Such a conjecture was originally formulated by Weinstein in the late 1970's 
\cite{Weinstein:1979} under the additional assumption that the cohomology do not vanish in degree one, and has received since then great attention. Nowadays, 
the conjecture is known to be true in dimension 3 \cite{Taubes:2007}; in higher dimension, the conjecture is proved only in special cases.
Theorem \label{thm:2} above can therefore be seen as a confirmation of the Weinstein conjecture for certain contact type hypersurfaces in cotangent bundles. 
To our best knowledge, the full Weinstein conjecture in cotangent bundles seems not to be known. In contrast, it is known to hold
for compact contact type hypersurfaces in twisted cotangent bundles $(T^*M,\omega_\std - \pi^*\sigma)$, provided
the closed two-form $\sigma$ does not vanish on $\pi_2(M)$; see \cite{Merry:2020}.

\vspace{2mm}

Theorem \ref{thm:2} will be an immediate consequence of a nearby/dense existence theorem of closed leaves for $\mathbb O_M$-separating hypersurfaces  which are not 
necessarily of contact type. Roughly speaking, if the contact condition is dropped, then one cannot expect the 
existence of closed characteristic leaves on $\Sigma$, as many explicit examples show (see e.g. \cite{Ginzburg:1996,Ginzburg:1997,Ginzburg:2003}).
However, one might hope to find closed characteristic leaves on hypersurfaces which are arbitrarily close to $\Sigma$.
To set the notation we define, following a suggestion of Kai Zehmisch, a \textit{thickening} of $\Sigma$ to be a
diffeomorphism $\Psi : (-a,a)\times \Sigma \to T^*M$, $a\in \R\cup \{+\infty\}$, onto an open precompact neighborhood $U\subset T^*M$ of $\Sigma$ such that 
$\Psi(0,\cdot) = \imath_\Sigma:\Sigma \to T^*M$ canonical inclusion. For every $\sigma \in (-a,a)$, we set $\Sigma_\sigma := \Psi (\{\sigma\}\times \Sigma)$, and denote with $\mathcal P(\sigma)$ the set of closed characteristic leaves contained in $\Sigma_\sigma$.
Notice that, if $\Sigma$ is regular and $\mathbb O_M$-separating, then up to shrinking the interval $(-a,a)$ 
we can assume that each $\Sigma_\sigma$ is regular and $\mathbb O_M$-separating. 
Also, every thickening can be realized as the flow of some vector field on $T^*M$ which is transverse 
to $\Sigma$. 

\begin{thm}
Let $\Sigma\subset T^*M$ be a compact, connected, $\mathbb O_M$-separating hypersurface, and let $\Psi$ be a thickening of $\Sigma$. Then there exists a sequence 
$\sigma_n\to 0$ such that $\mathcal P(\sigma_n)\neq \emptyset$ for all $n\in \N$. Moreover, we can find a constant $\alpha=\alpha(\Psi)>0$ such that for every $n\in\N$ there
exists $P_n\in \mathcal P(\sigma_n)$ with 
$$0<\Big |\int_{P_n} \lambda_\std \Big |<\alpha.$$
\label{thm:3}
\end{thm}

\vspace{-3mm}

Our proof of Theorem \ref{thm:3} follows closely the original argument of Hofer-Viterbo, nevertheless the new functional setting will enable us to strongly 
simplify the argument in its key technical parts. Indeed, Hofer-Viterbo's setting corresponds in the notation above to the case $s=1$, and it is well-known that in this 
case the Hamiltonian action $\A_H$ does not satisfy the Palais-Smale condition, because of the lack of compactness in the Hamiltonian part of the functional. 
Therefore, one has to introduce approximations of $\A_H$ to achieve compactness, and 
then pass to the limit for the approximations going to zero using a very delicate diagonal argument. 
In our case instead we can work directly with the functional $\A_H$, see Section \ref{section:3}. 

\vspace{2mm}

\textbf{Structure of the paper.} In Section 2, we introduce the necessary background on the Hamiltonian action $\A_H$ and on the
functional setting, and prove Theorem \ref{thm:1}. In Section 3, we show how Theorems 1.2 and \ref{thm:3} follow from an existence theorem 
of critical points for $\A_H$, which will be then proved in Section 4.

\vspace{2mm}

\textbf{Acknowledgments.} The authors warmly thank Alberto Abbondandolo, Thomas Bartsch, Marek Izydorek, and Kai Zehmisch for many fruitful discussions.
Starting point for this paper were lectures given by the first named author at the Justus-Liebig Universit\"at Gie\ss en, Germany, and at the Politechnika Gdanska, Poland,
on the classical paper by Hofer and Viterbo \cite{Hofer:1988}. L.A. warmly thanks Marek Izydorek and Joanna Janczewska 
for their kind hospitality. This research is supported by the DFG-project ``Morse theoretical methods in Hamiltonian dynamics''.
L.A. is partially supported
by the DFG-grant CRC/TRR 191 ``Symplectic structures in Geometry, Algebra and Dynamics".
M.S. is partially supported by the Beethoven2-grant  2016/23/G/ST1/04081 of the National Science Centre, Poland.


\section{The Hamiltonian action functional}
\label{section:2}

In this section, we introduce the functional setting for the Hamiltonian action $\A_H$ in \eqref{hamactiongeneral} on the cotangent bundle $T^*M$ of 
a closed manifold $M$ and prove Theorem \ref{thm:1}. 
We start recalling some well-known facts about Riemannian metrics on $M$ which will be useful later on.


\subsection{Bumpy metrics} 
A Riemannian metric $g$ yields a flow on $TM$ (the \textit{geodesic flow}) by 
$$TM \ni (q,v)\mapsto (\gamma(t),\dot \gamma(t)), \quad \forall t\in \R,$$
where $\gamma:\R\to M$ is the unique curve satisfying 
$$\nabla_{\dot \gamma} \dot \gamma =0, \quad \text{and}\ \gamma(0)=q, \ \dot \gamma(0)=v.$$
Here, $\nabla_{\dot \gamma}$ denotes the covariant derivative along $\gamma$ associated with the Levi-Civita connection.
The curve $\gamma$ is called the \textit{geodesic} through the point $q$ with initial velocity $v$. It is well-known that periodic orbits of the 
geodesic flow are in one-to-one correspondence with the critical points of the \textit{energy functional}
$$\E:H^1(\T, M) \to \R,\quad \E(\gamma) := \frac 12 \int_0^1 |\dot \gamma(t)|^2 \, \diff t,$$
where $|\cdot|:= \sqrt{g_{\gamma(t)}(\cdot,\cdot)}$ is the norm induced by the Riemannian metric, and $H^1(\T,M)$ is the 
Hilbert manifold of loops in $M$ of class $H^1$, i.e. absolutely continuous loops with square integrable derivative. 
More details on the Hilbert manifold structure of $H^1(\T,M)$ and on the properties of the functional $\E$ 
can be found e.g. in \cite{Klingenberg:1978so} (see also \cite{Abbondandolo:2009gg}). Here we just recall that the functional $\E$ satisfies the \textit{Palais-Smale condition}
, meaning that any sequence $(\gamma_n)\subset H^1(\T,M)$ 
such that 
$$\E(\gamma_n)\to e, \quad |\diff \E (\gamma_n)|\to 0
,$$ admits a converging subsequence.  In particular, $e$ is a critical value of $\E$. 
The next lemma is certainly well-known to the experts, however we include its proof here for the reader's convenience.

\begin{lem}
Let $M$ be a closed manifold. Then there exists a Riemannian metric $g$ on $M$ such that the set of critical values of the associated energy functional is discrete. 
\label{lembumpy}
\end{lem}
\begin{proof}
Notice first that, for \textit{any} Riemannian metric on $M$, zero is an isolated critical value for $\E$. Indeed, zero is a critical value 
since the set of constant loops $\Lambda^0M \cong M$ is the (non-degenerate\footnote{A critical manifold $\mathcal C$ for $\E$ is called \textit{non-degenerate} if the 
nullity of the Hessian of $\E$ at any $\gamma\in \mathcal C$ equals the dimension of $\mathcal C$.}; 
c.f. \cite[Proposition 2.4.6]{Klingenberg:1978so}) critical manifold of global minima for $\E$, 
and on the other hand it is isolated because of the existence of a positive injectivity radius. Actually,  for $\epsilon>0$ sufficiently small the set 
$\Lambda^0M$ is a strong deformation retract of $ \E^{-1}([0,\epsilon))$; see \cite[Theorem 1.4.15]{Klingenberg:1978so}. 

A standard result in Riemannian geometry, orginally proved by Abraham \cite{Abraham:1970} (see also \cite{Anosov:1983}), asserts that the set of Riemannian metrics 
on $M$ all of whose closed geodesics are non-degenerate (that is, the set of \textit{bumpy metrics}) is residual in the set of all Riemannian metrics. Thus, 
pick one such bumpy metric $g$, 
and let $e\in [0,+\infty)$ be a critical value for the corresponding energy functional $\E$. By the discussion above we can assume that $e>0$. Since $\E$ satisfies 
the Palais-Smale condition, the set crit$\, (\E)\cap \E^{-1}(e)$ is compact. Moreover, in virtue of the Morse Lemma for the functional $\E$ (c.f. \cite[Corollary 2.4.8]{Klingenberg:1978so}), 
any connected component of crit$\, (\E)\cap \E^{-1}(e)$ must be an isolated critical manifold. In particular, crit$\, (\E)\cap \E^{-1}(e)$ consists of finitely many 
non-degenerate critical manifolds: indeed, suppose by contradiction that $K_1,K_2,...$ are the connected components of crit$\, (\E)\cap \E^{-1}(e)$, and for each 
$k\in \N$ pick $\gamma_k\in K_k$. Then, $(\gamma_k)\subset H^1(\T,M)$ is a Palais-Smale sequence for $\E$ and hence, up to extracting a subsequence, it must converge to 
some $\gamma\in$ crit$\, (\E)\cap \E^{-1}(e)$. Therefore, the sequence $(\gamma_k)$ must be eventually constant. 

Finally, since crit$\, (\E)\cap \E^{-1}(e)$ consists of finitely many critical manifolds, it follows again from \cite[Corollary 2.4.8]{Klingenberg:1978so} that $e$ is an isolated 
critical value of $\E$. 
\end{proof}



\subsection{The setting}
Let $M$ be a closed $n$-dimensional manifold. Hereafter we identify tangent and cotangent bundles of $M$ by means of the musical isomorphism 
$$\flat: TM \to T^*M, \quad X\mapsto \flat (X) := g_{\pi(X)}(X,\cdot)$$
induced by a fixed metric $g$ on $M$.  As we now recall, for $s>\frac 12$ the fractional Sobolev space $H^s(\T,M)$ of $H^s$-loops in 
$M$ has a natural structure of Hilbert manifold, and for any $r\in \R$ there exists a vector bundle 
$$\pi_r:\mathcal M^{r} \to H^s(S^1,M)$$
over $H^s(S^1,M)$, whose typical fiber is given by ``vector fields of regularity $H^r$'' along a smooth loop (for $r<0$ these are actually elements 
in the dual space). 

We denote by $|\cdot|_q:= \sqrt{g_q (\cdot,\cdot)}$ the norm induced by the Riemannian metric $g$ on $T_qM$. 
For $\q\in C^\infty(S^1,M)$, the metric $g$ induces an $L^2$-scalar product on the space $\Gamma(\q^*TM)$ of smooth vector fields along $\q$ by
$$\langle \cdot ,\cdot \rangle := \int_0^1 g_\q (\cdot ,\cdot) \, \diff t.$$
The induced norm will be denoted by $\|\cdot\|$ without further specifying the loop $\q$. Similarly, we denote by 
$$\|\cdot \|_\infty := \sup_{t\in [0,1]} |\cdot |_{\q(t)}.$$

\begin{lem}
\label{lem:2.2}
Let $\q\in C^\infty(S^1,M)$, and let $0\leq\lambda_0(\q)\leq \lambda_1(\q)\leq \lambda_2(\q)\leq ...$ be the sequence of ordered eigenvalues of the self-adjoint operator 
$$-\nabla_{\dot \q}^2 = \nabla_{\dot \q}^* \circ \nabla_{\dot \q} : \Gamma(\q^*TM) \to \Gamma(\q^*TM),$$
where $\nabla_{\dot \q}$ denotes the covariant derivative along $\q$ and $\nabla_{\dot \q}^*$ its adjoint operator. 
Then, there exists $d(\|\dot \q\|_\infty)>0$, and $c,C>0$ depending only on $g$ such that 
\begin{equation}
c \big (j^2 - d(\|\dot \q\|_\infty )\big ) \leq \lambda_j (\q) \leq C \big (j^2 + d(\|\dot \q\|_\infty)\big ), \quad \forall j\in \N,
\label{growtheigenvalues}
\end{equation}
Moreover, any eigenvector $\xi$ of $-\nabla_{\dot \q}^2$ with $\|\xi\|=1$ satisfies $\|\xi\|_\infty\leq \sqrt2$.
\label{secondcovariantderivative}
\end{lem}
\begin{proof}
See Appendix \ref{app:A}.
\end{proof}

For $\q\in C^\infty(S^1,M)$ we denote by $\{\lambda_j(\q)\}_{j\in\N}$ the set of ordered eigenvalues of $\nabla_{\dot \q}^* \circ \nabla_{\dot \q}$, and with $\{\xi_j(\q)\}_{j\in\N}$
the corresponding set of orthonormal eigenvectors. For all $r\geq 0$ we set 
$$H^r(\q^*TM) := \Big \{ \p = \sum_{j=1}^{+\infty} p_j \xi_j(\q) \in L^2(\q^*TM) \ \Big |\ \sum_{j=1}^{+\infty}(1+\lambda_j(\q))^r |p_j|^2 <+\infty\Big \},$$
and denote with $H^{-r}(\q^*TM):=(H^r(\q^*TM))^*$ the dual space to $H^r(\q^*TM)$. Notice that we can interpret elements in $H^{-r}(\q^*TM)$ as formal series: 
$$H^{-r}(\q^*TM) = \Big \{ \p = \sum_{j=1}^{+\infty} p_j \xi_j(\q)\ \Big |\ \sum_{j=1}^{+\infty}(1+\lambda_j(\q))^{-r} |p_j|^2 <+\infty\Big \}.$$
The self-adjoint operator $\nabla_{\dot \q}^* \circ \nabla_{\dot \q}$ might have non-trivial (though finite dimensional) kernel, which is namely generated by 1-periodic parallel vector fields 
along $\q$. We set 
$$N(\q):=  \dim \ker( \nabla_{\dot \q}^* \circ \nabla_{\dot \q})\in \{0,...,n\},$$ so that $\lambda_1(\q)=...=\lambda_{N(\q)}(\q)=0$
and $\lambda_j(\q)>0$ for $j>N(\q)$, and define 
\begin{equation}
\langle \xi,\zeta\rangle_r := \sum_{j\in \N}^{+\infty} (1+\lambda_j(\q))^r \ \xi_j\zeta_j.\label{Hrscalarproduct}
\end{equation}
We also define for $r\in\R$ the operator $A^r=A^r(\q):=(1+\nabla_{\dot \q}^*\nabla_{\dot \q})^{r/2}$ by 
$$A^r : H^r(\q^*TM)\to L^2(\q^*TM), \quad A^r\Big (\p = \sum_{j=1}^{+\infty} p_j \xi_j(\q)\Big ) := \sum_{j=1}^{+\infty} (1+\lambda_j(\q))^{r/2} p_j \xi_j(\q),$$
so that $\|A^r \p\|_{2}= \|\p\|_r$ holds for all $\p\in H^r(\q^*TM)$. Notice that, by Lemma \ref{secondcovariantderivative} we have that:
\begin{itemize}
\item for all $r>r'$, the inclusion $H^r(\q^*TM) \to H^{r'}(\q^*TM)$ is continuous and compact, and
\item for all $r>\frac 12$, the inclusion $H^r(\q^*TM)\to C^0(\q^*TM)$ is continuous and compact.
\end{itemize}

\begin{lem}
For every $r\in\R$ the operator $A^r$ commutes with $\nabla_{\dot \q}$. 
\label{commute}
\end{lem}
\begin{proof}
It suffices to check that 
$$(A^{r} \circ \nabla_{\dot \q}) \xi_j(\q) = (\nabla_{\dot \q} \circ A^{r})  \xi_j(\q), \quad \forall j\in \N.$$
By definition we have that 
$$A^{r}( \xi_j(\q)) = (1+\lambda_j(\q))^{r/2} \xi_j(\q)$$ 
and hence 
$$(\nabla_{\dot \q} \circ A^{r}) \xi_j(\q) = (1+\lambda_j(\q))^{r/2} \nabla_{\dot \q}\xi_j(\q).$$
On the other hand $\nabla_{\dot \q} \xi_j(\q)$ is again an eigenvector for $-\nabla_{\dot \q}^2$ corresponding to the eigenvalue $\lambda_j(\q)$, and hence
\begin{equation*}
(A^{r} \circ \nabla_{\dot \q}) \xi_j(\q) =  (1+\lambda_j(\q))^{r/2} \nabla_{\dot \q}\xi_j(\q). \qedhere
\end{equation*}
\end{proof}

For every $q\in M$ we denote by $\exp_q :T_qM \to M$ the exponential map, and choose $\epsilon>0$ smaller than the 
injectivity radius of $M$. For every $\q\in C^\infty(S^1,M)$ let $H^s(\q^*\mathbb O_\epsilon)\subset H^s(\q^*TM)$ be the space of $H^s$-vector fields 
along $\q$ whose image is entirely contained in the $\epsilon$-ball around the zero-section of $\q^*TM$, 
and define 
$$\text{Exp}_\q : H^s(\q^*\mathbb O_\epsilon) \to \U^s_\q, \quad \xi \mapsto \text{Exp}_\q(\xi) (t) := \exp_{\q(t)} (\xi (t)).$$
Following \cite[Sections 1.2-1.3]{Klingenberg:1978so}, the differentiable structure on $H^s(S^1,M)$ is given by declaring the collection 
$\{(\U^s_\q, (\text{Exp}_\q)^{-1})\}$ to be an atlas of $H^s(S^1,M)$. As it turns out, the inclusions
$$C^0(S^1,M)\hookrightarrow H^s(S^1,M) \hookrightarrow C^\infty(S^1,M)$$
are continuous homotopy equivalences. Extending the definition of $H^r(\q^*TM)$ to any loop in $H^s(S^1,M)$ by mean of the differential of the map $\text{Exp}_\q$ 
yields now the desired vector bundle $\pi_r:\M^r\to H^s(S^1,M)$.
Such a bundle carries a natural Riemannian metric, which on the typical fiber is given by \eqref{Hrscalarproduct}. We denote this metric again 
with $\langle \cdot,\cdot\rangle_r$, and observe that it can be equivalently written as 
$$\langle \xi,\zeta\rangle_r = \int_0^1 g_\q \big ( (\text{id}+\nabla_{\dot \q}^* \circ \nabla_{\dot \q})^r \xi, \zeta \big )\, \diff t = \langle (\text{id} + \nabla_{\dot \q}^* \circ \nabla_{\dot \q})^r\xi,\zeta\rangle.$$


For our purposes, it will be convenient to define another metric for the bundle $\pi^r$, which will be denoted by  
$\langle \cdot ,\cdot \rangle^\emb_r$; as it turns out, $\langle \cdot ,\cdot \rangle^\emb_r$ is equivalent to $\langle \cdot,\cdot\rangle_r$ on every bundle chart, 
thus on every bounded set (see Lemma \ref{lem:equivalence}), but in general the two metrics are not globally equivalent (see Appendix \ref{app:B}). 
To define $\langle \cdot ,\cdot \rangle^\emb_r$ we proceed as follows: 
By the isometric embedding theorem of Nash-Moser, $(M,g)$ admits an isometric embedding into 
$\R^N$ for some $N\in \N$ large enough. This yields an equivalent definition of 
$$H^s(S^1,M) := \Big \{ u \in H^s(S^1,\R^N) \ \Big |\ u(\cdot) \subset M\Big \},$$
as well as a scalar product $\langle \cdot ,\cdot \rangle_r^\emb$ on $\Gamma(\q^*TM)$ for every $r\geq0$ and every $\q\in C^\infty(S^1,M)$:
\begin{equation}
\langle \xi,\zeta \rangle_r^{\emb} := \int_0^1 g_\q ( (\text{id} + \Delta)^r \xi,\zeta)\, \diff t,
\label{emb}
\end{equation}
where $\Delta \xi := \ddot \xi$. As usual, we denote the extension of \eqref{emb} to any loop in $H^s(S^1,M)$ again with $\langle \cdot,\cdot\rangle_r^\emb$. 

For $\q \in C^\infty(S^1,M)$ we set 
$$L_0:=1+\Delta,\quad L_1:=1+\nabla^*_{\dot \q} \circ \nabla_{\dot \q}.$$ 
The operators $L_0$ and $L_1$ are self-adjoint and positive, and clearly $L_0\geq L_1$, meaning that the difference $L_0-L_1$ is a positive operator. 
It is a result known as the \textit{L\"owner-Heinz theorem} \cite{Heinz:1951} (see also Kato \cite{Kato:1952}) 
that the function $f(t)= t^r$ is, for every $r\in [0,1]$, operator monotone over the interval $(0,+\infty)$, meaning that if $A\geq B$ then $A^r\geq B^r$.
This implies that $L_0^r\geq L_1^r$ for all $r\in [0,1]$. Therefore, 
since the function $t\mapsto - t^{-1}$ is operator monotone too \cite{Furuta}, we obtain that $L_0^{-r} \leq L_1^{-r}$, which is equivalent to saying that
\begin{equation}
\|\cdot\|_{-r}^\emb \leq \|\cdot \|_{-r},\quad \forall r\in [0,1].
\label{inequalityoperators}
\end{equation}

Recall that a sequence $(\q_n)$ is \textit{bounded in} $H^s(S^1,M)$ if there exists $c>0$ such that 
$$\|\dot \q_n\|_{s-1} \leq c ,\quad \forall n\in \N.$$
\begin{lem}
Let $(\q_n)$ be a bounded sequence in $H^s(S^1,M)$. Then up to passing to a subsequence we have that $\q_n\to \q \in C^0(S^1,M)$ uniformly.
\label{lem:aa}
\end{lem}
\begin{proof}
We see $(\q_n)$ as a sequence in $H^s(S^1,\R^N)$. By \eqref{inequalityoperators} we have that 
$$\|\dot \q_n\|_{s-1}^\emb \leq c, \quad \forall n\in \N.$$
Therefore, 
$$\|\q_n\|_s^\emb \leq \|\q_n\|_2 + \|\dot \q_n\|_{s-1}^\emb \leq \tilde c, \quad \forall n\in \N,$$
for some constant $\tilde c>0$, where we used the fact that $M$ is compact. In particular, the sequence $(\q_n)\subset H^s(S^1,\R^N)$ is $(s-\frac12)$-H\"older equicontinuous
\cite[Theorem 8.2]{Dinezza:2012}, and since $\q_n(\cdot)\subset M$ for all $n\in\N$, this implies that the hypothesis of the Ascoli-Arzel\'a theorem are satisfied. Therefore, there exists $\q\in C^0(S^1,\R^N)$ such that $\q_n\to \q$ uniformly. Now, by pointwise convergence we readily see that  $\q\in C^0(S^1,M)$.
\end{proof}

We finish this section showing that the metrics $\langle \cdot ,\cdot \rangle^\emb_r$ and $\langle \cdot,\cdot\rangle_r$ are equivalent on every bundle 
chart, and hence on every bounded set $B\subset H^s(S^1,M)$.

\begin{lem}
Let $\mathrm{Exp}_\q:H^s(\q^*\mathbb O_\epsilon)\to \U^s_\q$ be the local parametrization of $H^s(S^1,M)$ around $\q\in C^\infty(S^1,M)$. 
Then, for every $\gamma \in \U^s_\q$ the scalar products $\langle \cdot ,\cdot \rangle^\emb_r$ and $\langle \cdot ,\cdot \rangle_r$ are equivalent on $H^r(\gamma^*TM)$.
As a corollary, for every $B\subset H^s(S^1,M)$ bounded, the metrics $\langle \cdot ,\cdot \rangle^\emb_r|_B$ and $\langle \cdot ,\cdot \rangle_r|_B$ are equivalent.
\label{lem:equivalence}
\end{lem}
\begin{proof}
Let $\q\in C^\infty(S^1,M)$. By \cite[Proposition 5.6.1]{Masiello}, there exists a constant $\epsilon >0$ such that 
$L_1\geq \epsilon L_0$, which in virtue of the Heinz-Loewner theorem implies that 
$$\epsilon^r L_0^r \leq L_1^r \leq L_0^r,\quad \forall r\in [0,1],$$
that is, that the scalar products $\langle \cdot ,\cdot \rangle^\emb_r$ and $\langle \cdot ,\cdot \rangle_r$ are equivalent in $H^r(\q^*TM)$. 

Write now $\gamma\in \U^s_\q$ as $\gamma=\text{Exp}_\q(\xi)$. The assertion follows from the fact that the local representation of the metric 
$\langle \cdot ,\cdot \rangle^\emb_r$ resp. $\langle \cdot ,\cdot \rangle_r$ of $H^r(\text{Exp}_\q(\xi)^*TM)$ in $H^r(\q^*TM)$ is equivalent to the Hilbert metric 
$\langle \cdot ,\cdot \rangle^\emb_r$ resp. $\langle \cdot ,\cdot \rangle_r$ in $H^r(\q^*TM)$ (see the proof of Theorem 1.4.5 in \cite{Klingenberg:1978so}), 
combined with the fact that $\langle \cdot ,\cdot \rangle^\emb_r$ and $\langle \cdot ,\cdot \rangle_r$ are equivalent in $H^r(\q^*TM)$.

The equivalence of the metrics on bounded sets follows now immediately from the fact that every bounded set $B\subset H^s(S^1,M)$ can be covered
by finitely many local charts. This follows from Lemma 2.4; the details are left to the reader.
\end{proof}


\subsection{The Palais-Smale condition}
As in the previous section, let $(M,g)$ be a closed Riemannian manifold.
For $s\in (\frac 12,1]$ we consider the Hilbert-bundle $\pi_{1-s}:\mathcal M^{1-s}\to H^s(S^1,M)$. 
Given a smooth time-depending Hamiltonian function $H:\T\times TM\to \R$ such that 
$$H(t,q,p) = \frac 12 |p|_q^2, \ \forall t\in \T,$$
outside a compact set $K\subset TM$, we can define the Hamiltonian action functional by
\begin{align*}
\A_H : \mathcal M^s \to \R,\quad \A_H(\q,\p) &:= \int_0^1 g_{\q} ( \dot\q(t), \p(t))\, \diff t - \int_0^1 H(t,\q(t),\p(t))\, \diff t\\
								&= \langle \dot \q,\p\rangle - \frac 12 \|\p\|^2 - \int_0^1 \delta(t,\q(t),\p(t))\, \diff t,
								\end{align*}
where $\delta:TM \to \R,\ \delta(q,p) = H(t,q,p)-\frac 12 |p|_q^2$, is a smooth compactly supported function. We also set 
\begin{equation}
\Delta:\mathcal M^s \to \R, \quad \Delta (\q,\p) := \int_0^1 \delta(t,\q(t),\p(t))\, \diff t.
\label{cc}
\end{equation}


To see that $\A_H$ is well-defined and of class $C^{1,1}$ on $\M^{1-s}$, we embed $M$ isometrically into $\mathbb{R}^N$. This induces an embedding of 
$TM$ into $\mathbb{R}^{2N}$, as well as an embedding of 
$\mathcal{M}^{s-1}$ into $\mathcal E:=H^{s}(S^1,\mathbb{R}^N) \times H^{1-s}(S^1,\mathbb{R}^N)$. 
We now extend $\mathbb{A}_H$ to $\mathcal E$ by
extending $\langle\dot{\q},\p\rangle$ with the same formula, and $H:TM \to \mathbb{R}$ to any smooth Hamiltonian on $\mathbb{R}^{2N}$ which is quadratic at infinity. 
On $T\M^{1-s}$ we consider the splitting into horizontal and vertical subbundles induced by the $L^2$-connection, which is nothing else but the Levi-Civita connection 
applied pointwise. Notice that such a splitting coincides with the splitting that one naturally obtains by embedding $\M^{1-s}$ into $\mathcal E$. Denoting with 
$\xi^\hor$ and $\xi^\ver$ respectively the horizontal and vertical part of a tangent vector $\xi\in T_{(\q,\p)} \M^{1-s}$, we define a Riemannian metric on $\mathcal M^{1-s}$ by
\begin{equation}
\langle \cdot,\cdot \rangle_{\mathcal M^{1-s}} := \langle \cdot^\hor ,  \cdot^\hor \rangle_s + \langle  \cdot^\ver ,\cdot^\ver \rangle_{1-s}.
\label{metricms}
\end{equation}

Following \cite[Section 3.3]{Hofer:1994bq}, and using the fact that the gradient of the restriction is the projection of the gradient, we obtain  

\begin{lem}
\label{lem:compactness}
$\A_H$ is well-defined over $\mathcal M^{1-s}$ and of class $C^{1,1}$.  Moreover, for $s\in (\frac 12 ,1)$ the operator $\diff \Delta$ is compact. Finally, 
critical points of $\A_H$ correspond to one-periodic solutions of Hamilton's Equation \eqref{hamsystem}. \qed
\end{lem}

We shall mention that, for $s\in (\frac 12,1)$, $\A_H$ is actually more regular than $C^{1,1}$ even though it is in general not smooth. More precisely, arguing as in Appendix A.3 in \cite{Hofer:1994bq}
one can see that for every $s\in (\frac 12, 1)$ there exists $k=k(s)\in \N$ such that $\A_H:\M^{1-s}\to \R$ is of class $C^k$, with $k(s)\to +\infty$ 
as $s\downarrow \frac 12$.

\vspace{2mm}

We recall that a sequence $(\q_n,\p_n)\subset \mathcal M^{1-s}$ is called a \textit{Palais-Smale sequence} for $\A_H$ if $\A_H(\q_n,\p_n)\to a$ for some $a\in \R$ and $\|\diff \A_H(\q_n,\p_n)\|\to 0$. Without loss of generality we can assume that both $\q_n$ and $\p_n$ are smooth. Here, with slight abuse of notation we denote with $\|\cdot\|$ the dual norm on $T^*_{(\q_n,\p_n)}\mathcal M^{1-s}$ induced by the Riemannian metric 
$\langle \cdot,\cdot\rangle_{\mathcal M^{1-s}}$ given by \eqref{metricms}. We are now in position to prove Theorem \ref{thm:1}, which we reformulate for the reader's 
convenience with the following

\begin{prop}
For every $s\in (\frac 12,1)$ the functional $\A_H:\M^{1-s}\to \R$ satisfies the Palais-Smale condition.
\label{prop:ps}
\end{prop}

The key step to prove the proposition is the following

\begin{lem}
Let $(\q_n,\p_n)$ be a Palais-Smale sequence for $\A_H$. Then there exists a constant $C>0$ such that $\|\p_n\|_{1-s}\leq C$ and $\|\dot \q_n\|_{s-1}\leq C$ for all $n\in \N$. 
\label{lem:bdd}
\end{lem}
\begin{proof}[Proof of Lemma \ref{lem:bdd}]
We divide the proof in several steps. 

\textbf{Step 1.} $\|\dot\q_n\|_{s-1}$ is uniformly bounded iff $\|\p_n\|_{s-1}$ is uniformly bounded. For any $\vv_n\in H^{1-s}(\q_n^*TM)$ with $\|\vv_n\|_{1-s}\leq 1$ we compute 
\begin{align*}
o(1) &= \Big |\diff \A_H(\q_n,\p_n) [0,\vv_n]\Big | \\
	&= \Big |\langle \dot \q_n - \p_n, \vv_n\rangle - \int_0^1 \partial_p \delta(t,\q_n(t),\p_n(t)) \cdot \vv_n \, \diff t\Big|\\
	&\geq \big |\langle \dot \q_n - \p_n, \vv_n\rangle\big | - c \|\vv_n\|\\
	&\geq \big |\langle \dot \q_n - \p_n, \vv_n\rangle\big | - c
\end{align*}
and hence 
\begin{equation}
\| \jmath_{1-s}^* (\dot \q_n-\p_n)\|_{1-s} \leq c,
\label{verticalgradient}
\end{equation}
where $\jmath_{1-s}^*:L^2(\q_n^*TM)\to H^{1-s}(\q_n^*TM)$ is the adjoint operator to the inclusion $\jmath_{1-s}:H^{1-s}(\q_n^*TM)\to L^2(\q_n^*TM)$.
A straightforward computation shows that 
$$\jmath_{1-s}^* \Big ( v = \sum_{j=1}^{+\infty} v_j \xi_j(\q_n) \Big ) = \sum_{j=1}^{+\infty} (1+\lambda_j(\q_n))^{s-1} v_j \xi_j(\q_n),$$
that is, $\jmath_{1-s}^*=(1+\nabla_{\dot \q_n}^*\nabla_{\dot \q_n})^{s-1}.$ 
Moreover, with $\displaystyle \dot \q_n=\sum_{j=1}^{+\infty} \dot \q_n^j \xi_j(\q_n)$ we obtain
$$\|\jmath_{1-s}^*\dot\q_n\|_{1-s}^2 = \left \| \sum_{j=1}^{+\infty} (1-\lambda_j(\q_n))^{s-1} \dot \q^j_n \xi_j(\q_n)\right \|_{1-s}^2 =  \sum_{j=1}^{+\infty} (1-\lambda_j(\q_n))^{s-1} |\dot \q^j_n|^2 = \|\dot \q_n\|_{s-1}^2,$$
and similarly $\|\jmath_{1-s}^*\p_n\|_{1-s}=\|\p_n\|_{s-1}$. The claim follows from \eqref{verticalgradient}.

\vspace{2mm}

\textbf{Step 2.} $\|\p_n\|^2 \leq c (1+\|\p_n\|_{1-s}).$ We compute 
\begin{align*}
a + c \|\p_n\|_{1-s} & \geq \A_H(\q_n,\p_n) - \diff \A_H (\q_n,\p_n)[(0,\p_n)]\\
			     &= \frac 12 \|\p_n\|^2 - \int_0^1 \partial_p \delta (t,\q_n(t),\p_n(t)) \cdot \p_n\, \diff t + \int_0^1 \delta (t,\q_n(t),\p_n(t))\, \diff t\\
			     & \geq \frac 12 \|\p_n\|^2 - c(\|\p_n\| + 1)
\end{align*}
which implies the claim.

\vspace{2mm}

\textbf{Step 3.} $\|\nabla_{\dot \q_n} \p_n\|_{-s}$ is uniformly bounded. We compute for $\hhh_n\in H^s(\q_n^*TM)$:
\begin{align*}
c \|\hhh_n\|_s & \geq \Big | \diff \A_H (\q_n,\p_n) [(\hhh_n,0)]\Big |\\
		&= \Big | \langle \nabla_{\dot \q_n} \hhh_n,\p_n\rangle - \int_0^1 \partial_q \delta (t,\q_n(t),\p_n(t))\cdot \hhh_n \, \diff t\Big |\\
		& \geq  \Big | \langle \nabla_{\dot \q_n} \hhh_n,\p_n\rangle \Big | - c \|\hhh_n\|_s
\end{align*}
from which we deduce that 
$$\Big | \langle \nabla_{\dot \q_n} \hhh_n,\p_n\rangle \Big | \leq c \|\hhh_n\|_s.$$
Setting $\hhh_n := ((1+\nabla_{\dot \q_n}^*\nabla_{\dot \q_n})^{-s} \circ \nabla_{\dot \q_n} )\p_n$ and using Lemma \ref{commute} we obtain 
$$\|\nabla_{\dot \q_n} \p_n\|_{-s}^2 \leq c \|\nabla_{\dot \q_n}\p_n\|_{-s}$$
which readily implies the claim.

\vspace{2mm}

\textbf{Step 4.} $\|\p_n\|_{1-s}$ is uniformly bounded. We write $\p_n = \p_n^\parall + \tilde \p_n$, where $\p_n^\parall$ is the parallel component 
$$\p_n^\parall = \sum_{j=1}^{N(\q_n)} \p_n^j \xi_j(\q_n)$$
of $\p_n$  and  
$$\tilde \p_n := \sum_{j>N(\q_n)} \p_n^j \xi_j(\q_n).$$
Clearly, 
$$\|\p_n\|_{1-s} \leq \|\p_n^\parall\|_{1-s} + \|\tilde \p_n\|_{1-s} = \|\p_n^\parall\| + \|\tilde \p_n\|_{1-s},$$ 
where we have used the fact that $\|\p_n^\parall\|_{1-s}=\|\p_n^\parall\|$. In particular, it suffices to show that $\|\p_n^\parall\| $ and $\|\tilde \p_n\|_{1-s}$ are uniformly bounded.
We readily see that 
$$\|\nabla_{\dot \q_n} \p_n\|_{-s}^2 = \|\tilde \p_n\|_{1-s}^2 - \|\tilde \p_n\|^2,$$
and hence by Step 3 
\begin{equation}
 \|\tilde \p_n\|_{1-s}^2  \leq c (1+ \|\tilde \p_n\|^2).
 \label{first}
 \end{equation}
Step 2 implies now that 
$$\|\tilde \p_n\|^2 \leq \|\p_n\|^2 \leq c (1 + \|\p_n\|_{1-s} ) \leq c ( 1 + \|\p_n^\parall\| + \|\tilde \p_n\|_{1-s}).$$
Substituting in \eqref{first} yields 
$$\|\tilde \p_n\|_{1-s}^2\leq c (1 + \|\p_n^\parall\| + \|\tilde \p_n\|_{1-s} )$$
which implies 
\begin{equation}
\|\tilde \p_n\|_{1-s} \leq c (1 + \|\p_n^\parall\|^{1/2}).
\label{second}
\end{equation}
Using again Step 2 we obtain 
$$\|\p_n^\parall\|^2 \leq c (1+ \|\p_n^\parall\| + \|\p_n^\parall\|^{1/2})$$
which implies that $\|\p_n^\parall\|$, thus by \eqref{second} also $\|\tilde \p_n\|_{1-s}$, is uniformly bounded. 
\end{proof}

\begin{proof}[Proof of Proposition \ref{prop:ps}]
Let $(\q_n,\p_n)$ be a Palais-Smale sequence for $\A_H$. 
By Lemmas \ref{lem:aa} and \ref{lem:bdd}, up to extracting a subsequence we have that $\q_n\to \bar \q$ uniformly to some $\bar \q\in C^0(S^1,M)$. Therefore, 
up to neglecting finitely many $n$'s, we can suppose that all $(\q_n,\p_n)$ lie inside a bundle chart for $\mathcal M^{1-s}$ around a smooth loop $\q$,
where for every $r\in [-1,1]$ the metrics $\langle\cdot,\cdot\rangle_r$ and $\langle\cdot,\cdot\rangle_r^\emb$ are equivalent in virtue of Lemma \ref{lem:equivalence}. 

From the proof of Lemma \ref{lem:bdd}, Step 1, we see that 
$$o(1)=\big \| \jmath_{1-s}^* \big (\dot \q_n-\p_n - \text{Grad}_{L^2} \Delta (\q_n,\p_n)^\ver\big )\big \|_{1-s}^\emb,$$
where $\langle \text{Grad}_{L^2}\Delta(\q_n,\p_n)^\ver,\cdot \rangle = \diff_p \Delta (\q_n,\p_n) [\cdot]$ denotes the vertical part of the $L^2$-gradient of $\Delta$. 
Since $\diff \Delta$ is a compact operator (see Lemma \ref{lem:compactness}), up to a subsequence we have that $\jmath_{1-s}^*\text{Grad}_{L^2} \Delta (\q_n,\p_n)^\ver$ converges
in $H^{1-s}$. Therefore, $\jmath_{1-s}^*(\dot \q_n-\p_n)$ converges in $H^{1-s}$, which is the same as saying that $\dot \q_n-\p_n$ converges in $H^{s-1}$. 
 Now, $\p_n$ converges in $L^2$ (being bounded in $H^{1-s}$), and hence in particular converges in $H^{s-1}$. This implies that $\dot \q_n$ converges in $H^{s-1}$,
 which in turns yields the convergence of $\q_n$ in $H^s$.  
 
 On the other hand, from Step 3 in the proof of Lemma \ref{lem:bdd} we have that 
 $$o(1) = \big \| \jmath_s^* \big( \nabla_{\dot \q_n} \p_n - \text{Grad}_{L^2} \Delta(\q_n,\p_n)^\hor\big )\big \|_s^\emb,$$
 where $\jmath_s^*:L^2(\q_n^*TM)\to H^{s}(\q_n^*TM)$ is the adjoint operator to the inclusion $\jmath_s:H^{s}(\q_n^*TM)\to L^2(\q_n^*TM)$, and 
 $\langle \text{Grad}_{L^2}\Delta(\q_n,\p_n)^\hor,\cdot \rangle = \diff_q \Delta (\q_n,\p_n) [\cdot]$ denotes the horizontal part of the $L^2$-gradient of $\Delta$.
 Again, the compactness of $\diff \Delta$ yields that $\jmath_s^* \nabla_{\dot \q_n} \p_n$ converges (up to a subsequence) in $H^s$, which is equivalent to saying that 
 $\nabla_{\dot \q_n} \p_n$ converges in $H^{-s}$. This implies that, in the notation of the proof of Lemma \ref{lem:bdd}, $\tilde \p_n$ converges in $H^{1-s}$. Since 
 the kernel of $\nabla^*_{\dot \q_n}\circ \nabla_{\dot \q_n}$ is finite-dimensional, we also have that $\p_n^{\parall}$ converges up to a subsequence in $L^2$ (and hence in $H^{1-s})$. 
 Therefore, $\p_n$ converges in $H^{1-s}$.
\end{proof}


\section{Proof of Theorems 1.2 and \ref{thm:3}}
\label{section:3}

In this section we prove Theorems 1.2 and \ref{thm:3} on the existence of closed characteristic leaves for compact regular $\mathbb O_M$-separating 
hypersurfaces in cotangent bundles. To this purposes we will employ the correspondence between one-periodic Hamiltonian orbits and critical points of the 
Hamiltonian action $\A_H$.  As the Hamiltonian dynamics depends up to time reparametrization only on the hypersurface itself, we will choose a 
suitable one-parameter family of Hamiltonian functions, which we now construct, to perform the argument. 

\subsection{A special Hamiltonian function}
We choose a bumpy metric $g$ on $M$ and pull-back the standard symplectic form $\omega$ on $T^*M$ to $TM$ using the musical isomorphism. 
Given a compact regular $\mathbb O_M$-separating hypersurface $\Sigma\subset TM$ and a thickening $\Psi:(-a,a)\times \Sigma \to TM$, we aim at proving that 
there is a sequence of hypersurfaces $\Sigma_{\sigma_n}:=\Psi (\{\sigma_n\}\times \Sigma)$, $\sigma_n\to 0$, each carrying a closed characteristic leaf.

By assumption we can find $0<\rho_0<\rho_1<+\infty$ such that 
$$\U:= \Psi ((-a,a)\times \Sigma) \subset B_{\rho_1}(\mathbb O_M) \setminus B_{\rho_0}(\mathbb O_M),$$
where $B_\rho(\mathbb O_M)\subset TM$ denotes the open disk bundle with radius $\rho$ defined by $g$. 
We now fix $0<\delta <a$  and choose a cut-off function $\chi: (-1,1)\to \R$ such that 
$$\chi \equiv 0 \ \ \text{on}\ (-1,-\delta], \quad \chi \equiv 1 \ \ \text{on}\ [\delta,1), \quad \chi'>0 \ \ \text{on}\ (-\delta,\delta).$$
Furthermore, we pick a smooth function $\varphi:\R\to \R$ such that
$$\varphi \equiv 0 \ \ \text{on}\ (-\infty,\rho_1], \quad \varphi(\rho) = \frac 12 \rho^2 \ \ \text{on}\ [2\rho_1,+\infty), \quad \varphi'>0 \ \ \text{on} \ (\rho_1,+\infty)$$
and define a smooth family of Hamiltonians $H_r: TM \to \R$, $r>0$, by
$$H_r(q,p) :=
\left \{\begin{array}{r} 0 \, \qquad \qquad \qquad \quad \quad \quad X\in B,\\  \chi(\sigma)\cdot r \  \quad \quad X\in \Sigma_\sigma, \ \sigma\in [-\delta,\delta],\\ r \ \ \qquad \quad \quad X\in U\!B, \ |p|_q\leq \rho_1 ,\\ \varphi(|p|_q)+r \qquad \ \ \quad \quad \quad |p|_q >\rho_1,\end{array}\right.$$
where $B$ and $U\!B$ are the bounded and unbounded component of $TM\setminus \Psi([-\delta,\delta]\times \Sigma)$ respectively.
For each $r\in (0,+\infty)$ we have an associated Hamiltonian action
$$\A_r:=\A_{H_r}: \mathcal M^{1-s} \to \R,\quad \A_r (\q,\p) := \langle \dot \q,\p\rangle -\int_0^1 H_r(\q(t),\p(t))\, \diff t,$$
whose critical points are the 1-periodic orbits of the Hamiltonian flow defined by $H_r$ and $\omega$. However, not all 
critical points of $\A_r$ are relevant for us, for we are looking for critical points lying in $\Sigma_\sigma$ for some $\sigma\in [-\delta,\delta]$. Therefore, 
it will be essential for our purposes to understand which kind of critical points can appear as critical points of the Hamiltonian action $\A_r$.

Before doing that we shall observe that periodic orbits with period $T\neq 1$ for the Hamiltonian flow of $H_r$ which are contained in some 
$\Sigma_\sigma$ are detected as critical points of the Hamiltonian action $\A_{Tr}$.
Indeed, let $\x:\R/T\Z \to TM$ be a $T$-periodic Hamiltonian orbit for $H_r$ contained in $\Sigma_\sigma$, and consider the reparametrized curve 
$\tilde{\x}:\R/\Z\to TM, \ \tilde \x(t):=\x(Tt)$. Then 
$$\dot{\tilde \x}(t) = T\, \dot \x (Tt) = T X_{H_r} (\x (Tt)) = T X_{H_r}(\tilde \x(t)).$$
On the other hand, on $\Sigma_\sigma$ we have that 
$$H_r = \chi(\sigma) \cdot r, \quad H_{Tr} = \chi(\sigma) \cdot Tr,$$
so that $H_{Tr}=T\cdot H_r$ on $\Sigma_\sigma$. Therefore,
$$\dot{\tilde \x}(t) = T X_{H_r}(\tilde \x(t)) = X_{T\cdot H_r}(\tilde \x (t)) = X_{H_{Tr}}(\tilde \x(t)),$$
that is, $\tilde{\x}$ is a 1-periodic orbit for the Hamiltonian flow of $H_{Tr}$, and hence belongs to the critical point set of $\A_{Tr}$. This shows that 
the family of Hamiltonians $H_r$ detects all possible closed characteristic leaves contained in $\Sigma_\sigma$, for $\sigma\in [-\delta,\delta]$.

We now take a closer look at critical points of $\A_r$ by first noticing that critical points of $\A_r$ on non-regular energy levels are necessarily constant, 
and hence have non-positive $\A_r$-action. Also, regular energy levels $H_r^{-1}(a)$ are either of the form $\Sigma_\sigma$ for 
some $\sigma\in [-\delta,\delta]$, or (for $a>r$) sphere bundles over $M$, so that for every $a>r$ projected Hamiltonian orbits are geometrically
closed geodesics. However, the parametrizations do not coincide if $r<a<r+2\rho_1^2$ with the usual parametrizations of closed 
geodesics, as the Hamiltonian $H_r$ is not kinetic. We will refer to 
such critical points as \textit{fake closed geodesics}.
For $a\geq r+2\rho_1^2$ critical points of $\A_r$ contained in $H_r^{-1}(a)$ are instead of the form $(\gamma,\dot \gamma)$, for $\gamma$ closed geodesic on $(M,g)$ 
of length 1.  Indeed, for $a\geq r+2\rho_1^2$ we have that 
$$H_r(q,p) = \frac 12 |p|^2_q +r.$$ 
For any critical point $(\q,\p)$ of $\A_r$ contained in $H_r^{-1}(a), a\geq r+2\rho_1^2$, we additionally have
\begin{align}
\A_r(\q,\p) = \frac 12 |\p(0)|^2 - r = \frac 12 \int_0^1 |\dot \q(t)|^2\, \diff t - r = \E(\q) - r. \label{eqgeod}
			\end{align}
Our next step will be to show that, for $r$ sufficiently large, fake closed geodesics cannot arise as critical points of $\A_r$ with non-negative action. 
Indeed, Hamilton equations for fake closed geodesics read
\begin{equation}
\left \{\begin{array}{r} \dot \q = \displaystyle \frac{\varphi' (|\p|)}{|\p|} \cdot \p, \\ \ \ \ \ \quad  \, \nabla_{\dot \q}\p = 0 .\end{array}\right.
\label{parasitenc}
\end{equation}
Therefore, 
\begin{align}
\A_r(\q,\p) &= \langle \dot \q , \p\rangle - \int_0^1 H_r(\q(t),\p(t))\diff t \nonumber \\
			&= \langle \frac{\varphi' (|\p|)}{|\p|} \p,\p\rangle - \int_0^1 \big (\varphi(|\p(t)|) +r \big )\, \diff t\nonumber \\
			&= \varphi'(|\p(0)|) \cdot |\p(0)| - \varphi(|\p(0)|) -r,\label{eqparasite}
\end{align}
where we have used the fact that $t\mapsto |\p(t)|$ is constant. Now set 
\begin{equation}
r_0 := 1+ \max_{\rho\leq 2\rho_1} |\varphi'(\rho)\cdot \rho - \varphi(\rho)|
\label{b0}
\end{equation}
and observe that, for all $r\geq r_0$ and all fake closed geodesics  we have
$\A_r(\q,\p) \leq -1$, for $|\p(0)|\leq 2\rho_1$. Summarizing, we have shown the following

\begin{lem}
There exists $r_0>0$ such that for all $r\geq r_0$ critical points of $\A_r$ of non-negative action are either constants or closed geodesics, or are 
contained in $\Sigma_\sigma$ for some $\sigma\in [-\delta,\delta]$.
\label{lem:criticalpointsforlarger}
\end{lem}

We end this section showing that Theorems 1.2 and \ref{thm:3} immediately follow from 
\begin{thm}
Let $\Sigma\subset TM$ be a compact regular $\mathbb O_M$-separating hypersurface, $\Psi$ be a thickening of $\Sigma$. 
Then, for every $r>0$ there exists a non-constant critical point $(\q_r,\p_r)$ of $\A_r$ with $\A_r(\q_r,\p_r)\in [0,\alpha]$, where $\alpha=\alpha(\Psi)>0$ is 
some constant. Moreover, the function 
$r\mapsto \A_r(\q_r,\p_r)$ is continuous and non-increasing.
\label{thm:4}
\end{thm}

\begin{proof}[Proof of Theorem \ref{thm:3}]
Let $r_0$ be given by \eqref{b0}. By Lemma \ref{lem:criticalpointsforlarger} we can assume that all the critical points of $\A_r$, $r\geq r_0$, are closed geodesics with
$$\A_r(\q_r,\p_r) =  \E(\q_r)-r.$$
Since $g$ was chosen to be bumpy, by Lemma \ref{lembumpy} the set 
of critical values of $\E$ is discrete, and hence 
$$\A_r(\q_r,\p_r) + r = \E(\q_r) = \mathrm{const.} \quad \forall r\geq r_0.$$
However, this would imply that $\A_r(\q_r,\p_r)<0$ for $r$ large enough. Therefore, there exists $R\geq r_0$ 
such that $(\q_{R},\p_R)$ is a critical point for $\A_{R}$ lying in $\Sigma_\sigma$ for some $\sigma\in [-\delta,\delta]$. If $(\q_R,\p_{R})(\R)\subset \Sigma$ then we are done. Otherwise we claim that
$$\inf \big \{r\geq r_0 \ \big |\ (\q_r,\p_r) \in \Sigma_\sigma, \ \text{for some}\ \sigma\in [-\delta,\delta]\big \} \leq \alpha+r_0,$$
where $\alpha$ is the constant given by Theorem \ref{thm:4}. Indeed, for all $r\geq r_0$ smaller than the infimum above we have that 
$(\q_r,\p_r)$ is a closed geodesic and hence, using the uniform boundedness of $r\mapsto \A_r(\q_r,\p_r)$ and Lemma \ref{lembumpy}, we obtain 
$$\A_r(\q_r,\p_r)+r = \E(\q_r) = \E(\q_{r_0}) = \A_{r_0}(\q_{r_0}, \p_{r_0})+r_0 \leq \alpha + r_0$$
which implies that 
$$r\leq \alpha + r_0-\A_r(\q_r,\p_r) \leq \alpha+r_0.$$
In particular, we can find $R\leq \alpha+2r_0$ such that $(\q_{R},\p_R)$ lies in $\Sigma_\sigma$ for some $\sigma\in [-\delta,\delta]$. This yields
\begin{align}
\Big |\langle \dot \q_R,\p_R \rangle \Big |&= \Big |\A_R(\q_R,\p_R) - \int_0^1 H_R(\q_R(t),\p_R(t))\, \diff t\Big |\nonumber \\
						&= \Big |\A_R(\q_R,\p_R) - H_R(\q_R(0), \p_R(0))\Big |\nonumber \\
						&\leq \underbrace{\Big |\A_R(\q_R,\p_R)\Big |}_{\leq \alpha} + \underbrace{\Big |H_R(\q_R(0),\p_R(0))\Big |}_{\leq R\leq \alpha+2r_0} \leq 2 (\alpha+r_0). \label{estimatesymp}
\end{align}
The claim follows now by recursively choosing  $\delta>0$ such that
$$(\q_R,\p_R)(\R) \not \subset \Psi([-\delta,\delta]\times \Sigma).$$
Observe that \eqref{estimatesymp} yields the desired uniform estimate on the symplectic action of the sequence of closed characteristic leaves, for 
$$\langle \dot \q_R,\p_R\rangle =\int_{P_R}\lambda,$$
where $P_R$ is the characteristic leaf determined by $(\q_R,\p_R)$.
\end{proof}

The proof of Theorem 1.2 given Theorem \ref{thm:3} is standard, however we include it here for completeness. 

\begin{proof}[Proof of Theorem \ref{thm:2}]
Let $Y$ be a Liouville vector field
on a neighborhood of $\Sigma$ such that 
$Y \pitchfork \Sigma$, and let $\varphi^\sigma$ be its flow.
Since $\Sigma$ is compact, the map 
$$\Psi:(-a,a)\to T^*M,\quad (\sigma,x) \mapsto \varphi^\sigma(x),$$
is a diffeomorphism onto an open precompact neighborhood $U$ of $\Sigma$, for $a>0$ sufficiently small. From $\mathcal L_Y \omega=\omega$ we have that 
$$\ds (\varphi^\sigma)^* \omega  = (\varphi^\sigma)^* \mathcal L_Y \omega = (\varphi^\sigma)^* \omega$$
and hence, since $(\varphi^0)^*=\text{id}$, we conclude that $(\varphi^\sigma)^*\omega=e^\sigma\omega$. Assume now that $v\in \ell_\Sigma(x)$; then 
for all $w\in T_x\Sigma$ we have 
\begin{align*}
0 = \omega(v,w) = e^\sigma \, \omega(v,w) = (\varphi^\sigma)^* \omega(v,w) = \omega ( T\varphi^\sigma(x)[v],T\varphi^\sigma(x)[w]).
\end{align*}
Since $\varphi^\sigma$ is a diffeomorphism we conclude that $T\varphi^\sigma(x)[v]\in \ell_{\Sigma_\sigma(\varphi^t(x))}$. Therefore, $T\varphi^\sigma:\ell_\Sigma \to \ell_{\Sigma_\sigma}$ is an
isomorphism of line bundles; in particular, $\varphi^\sigma$ induces a one-to-one correspondence $P\mapsto \varphi^\sigma(P)$ between $\mathcal P(0)$ and $\mathcal P(\sigma)$ for all $\sigma\in (-a,a)$.
The claim follows now from Theorem \ref{thm:3}.
\end{proof}

\begin{rmk}
An hypersurface $\Sigma \subset T^*M$ for which a thickening as in the proof above exists is called \textit{stable}. Obviously, Theorem \ref{thm:2} extends to compact stable hypersurfaces which are $\mathbb O_M$-separating. It is worth noticing that the stability condition is in general weaker than the contact condition, see e.g. \cite{Cieliebak:2010zt}.
\end{rmk}


\section{Proof of Theorem \ref{thm:4}}

In this section we prove Theorem \ref{thm:4}. The proof is based on two key ingredients: one is essentially the Palais-Smale condition for the functional $\A_r$, the other is the fact that 
we have a transfer homomorphism in cohomology for the negative gradient flow of $\A_r$, as we now show. Hereafter we suppose that $r>0$ is fixed.


\subsection{The key propositions.} We start recalling the minimax lemma for the Hamiltonian action $\A_r$. The proof follows from the Palais-Smale condition for $\A_r$ by standard arguments and will be omitted.

\begin{prop}\label{deformationproposition}
Suppose that $\U\subset \M^{1-s}$ is an open neighborhood of 
$$\crit (\A_r) \cap \A_r^{-1}(a), \quad a\in \R.$$
Then there exist $\epsilon >0$ and $t_0>0$ such that the following holds: 
for every $t\geq t_0$
$$\phi^t_r \big (\{\A_r \leq a+\epsilon\}\setminus \U\big ) \subset \{ \A_r \leq a-\epsilon\},$$
where $\phi^t_r$ denotes the time-$t$-flow of $\displaystyle -\frac{\mathrm{grad}\, \A_r}{\sqrt{1+\|\mathrm{grad}\, \A_r\|^2}}.$\qed
\end{prop}

In what follows $C$ is an arbitrary compact subset of $H^1(S^1,M)\subset H^s(S^1,M)$. This  implies that  
$$\sup_{\pi^{-1}(C)} \A_r \leq \alpha,\quad \forall r>0,$$
where with slight abuse of notation we denote the bundle projection $\pi_{1-s}:\M^{1-s}\to H^s(S^1,M)$ with $\pi$. Here, $\alpha>0$ is some constant independent of $r$. Indeed, by construction we have 
$$H_r(q,p)\geq H_0(q,p) \geq \frac12 |p|_q^2 - \beta$$
for some constant $\beta>0$, and hence on $\pi^{-1}(C)$ we obtain
$$\A_r(\q,\p) \leq \langle \dot \q, \p\rangle - \frac 12 \|\p\|^2+\beta \leq c \|\p\| - \frac 12 \|\p\|^2+\beta \leq \sup_{\p \in \pi_{1-s}^{-1}(C)} \Big (c \|\p\| - \frac 12 \|\p\|^2 + \beta \Big ) =:\alpha.$$
Notice that if $C$ were compact in $H^s(S^1,M)$ but unbounded in $H^1(S^1,M)$ then the supremum above would be infinite. 
Since $\A_r$ satisfies the Palais-Smale condition, we can find $\epsilon>0$ and $\gamma>0$ such that 
$$\frac{\| \mathrm{grad}\, \A_r\| }{\sqrt{1+\|\mathrm{grad}\, \A_r\|^2}}\geq \epsilon, \quad \text{on}\ \{ \|\p\|_{1-s}\geq \gamma\} \cap \A_r^{-1}([0,\alpha]).$$
Therefore, for $\gamma':=\gamma + \frac{\alpha}{\epsilon^2}+1$ we have that 
$$\phi^t_r \Big (\pi^{-1}(C) \cap \{ \|\p\|_{1-s}\geq \gamma'\}\Big ) \cap \mathbb O_{H^s}=\emptyset,\quad \forall t\geq 0,$$
where $\mathbb O_{H^s}$ denotes the zero-section of $\M^{1-s}\to H^s(S^1,M)$. Indeed, let $(\q,\p)\in \pi^{-1}(C) \cap \{ \|\p\|_{1-s}\geq \gamma'\}$; 
then by the assumption on $\gamma'$, $\phi^t_r(\q,\p)\in H^s(S^1,M)\cap \{\|\p\|_{1-s}\geq \gamma\}$ for $t\in [0,\frac{\alpha}{\epsilon^2}+1]$, hence in particular 
is not contained in $\mathbb O_{H^s}$, and for $t>\frac{\alpha}{\epsilon^2}+1$ we have 
\begin{align*}
\A_r(\phi^t_r(\q,\p))-\alpha &\leq \A_r(\phi^t_r(\q,\p))-\A_r (\q,\p) \\
 					&= \int_0^t \frac{\diff }{\diff \sigma} \big (\A_r(\phi^\sigma_r(\q,\p))\big )\, \diff t\\
						&= - \int_0^t \frac{\|\text{grad}\, \A(\phi^\sigma_r(\q,\p))\|^2}{\sqrt{1+\|\text{grad}\, \A(\phi^\sigma_r(\q,\p))\|^2}}\, \diff t\\
						&\leq - \int_0^t \epsilon^2 \, \diff t\\
						&< - \Big (\frac{\alpha}{\epsilon^2}+1\Big )\epsilon^2	\\
						&= - \alpha - \epsilon^2,
\end{align*}
that is, $\A_r(\phi^t_r(\q,\p))<0$. For a given $t_0>0$ we pick a cut-off function $\varphi : [0,+\infty) \to [0,1]$ such that
$$\varphi\Big |_{[0,\gamma'+1]} \equiv 1,\quad \varphi\Big |_{[\gamma'',+\infty)} \equiv 0,$$
for some $\gamma''>\gamma'+1$ such that 
$$\phi^t_r\Big ( \pi^{-1}(C) \cap \{\|\p\|_{1-s}\leq \gamma'+1\}\Big ) \subset \{\|\p\|_{1-s}< \gamma''\}, \quad \forall t\in [0,t_0],$$
and consider the truncated normalized negative gradient vectorfield
$$V_r(\q,\p):= - \varphi(\|\p\|_{1-s})\cdot \frac{\mathrm{grad}\, \A_r(\q,\p)}{\sqrt{1+\|\mathrm{grad}\, \A_r(\q,\p)\|^2}}.$$
With a slight abuse of notation we denote the flow of $V_r$ again with $\phi^t_r$.
The next proposition states that $\phi^{t_0}_r$ induces a transfer homomorphism in 
cohomology; in particular, $\pi^{-1}(C)$ is not displaced from $\mathbb O_{H^s}$ by $\phi^{t_0}_r$. 
This represents the analogue of the intersection proposition \cite[Proposition 1]{Hofer:1988} in our setting; we also refer to \cite[Chapter 3, Lemma 10]{Hofer:1994bq} for an analogous statement in the linear setting. 
In what follows, $H^*$ denotes the Alexander-Spanier cohomology with coefficients in some given commutative ring.

\begin{prop}\label{intersectionproposition}
There exists an injective group homomorphism $\beta_{t_0}$ such that the following diagram commutes
$$\xymatrix{
H^*\big (\phi^{t_0}_r(\pi^{-1}(C))\cap \mathbb O_{H^s}\big ) & & \\
H^*(H^s(S^1,M)) \ar[u]^{\big (\pi \big |_{\phi^{t_0}_r(\pi^{-1}(C))\cap \mathbb O_{H^s}}\big )^*} \ar[rr]^{\imath^*} & & H^*(C)\ar[ull]_{\beta_{t_0}}
}
$$
where $\imath:C\to H^s(S^1,M)$ denotes the canonical inclusion. In particular, if $C\neq \emptyset$ then 
\begin{equation*}
 \phi^{t_0}_r(\pi^{-1}(C))\cap \mathbb O_{H^s}\neq \emptyset. \end{equation*}
\end{prop}

The rest of this subsection will be devoted to the proof of Proposition \ref{intersectionproposition}. The key ingredient of the proof will be a representation 
lemma for the flow $\phi^t_r$ analogous to \cite[Lemma 7]{Hofer:1988}. 

If we denote by $D$ the $L^2$-connection, then we readily see 
by working in local coordinates that 
$$[D_X,\nabla_{\dot \q}] Y = R(X,\dot \q) Y,$$
where $R$ denotes the Riemann curvature tensor, hence in particular is a zero-order operator. Therefore,
\begin{align*}
[D_X, - \nabla_{\dot \q}\nabla_{\dot \q}] Y &= - D_X \nabla_{\dot \q}\nabla_{\dot \q} Y + \nabla_{\dot \q}\nabla_{\dot \q} D_X Y\\
					&= -  \nabla_{\dot \q} D_X \nabla_{\dot \q} Y - R(X,\dot \q) \nabla_{\dot \q} Y +  \nabla_{\dot \q}\nabla_{\dot \q} D_XY\\
					&= - \nabla_{\dot \q}\nabla_{\dot \q} D_XY - \nabla_{\dot \q} R(X,\dot \q) Y - R(X,\dot \q) \nabla_{\dot \q} Y +  \nabla_{\dot \q}\nabla_{\dot \q} D_XY\\
					&= - \nabla_{\dot \q} R(X,\dot \q) Y - R(X,\dot \q) \nabla_{\dot \q} Y
					\end{align*}
is an operator of order 1. In particular
 $$[D_X, 1 - \nabla_{\dot \q}\nabla_{\dot \q}] = [D_X, - \nabla_{\dot \q}\nabla_{\dot \q}]$$
 is an operator of order 1. Similarly one can show that, for every $\ell\in \R$,  
 $$[D_X, (1 - \nabla_{\dot \q}\nabla_{\dot \q})^\ell] = [D_X, (1 + \nabla_{\dot \q}^*\nabla_{\dot \q})^\ell]$$
 is an operator of order at most $2\ell-1$ (c.f. \cite[Lemma 2.11]{Maeda:2015}). 

\begin{lem}[Representation Lemma]
Denote by $\sigma^t_r:= \pi\circ \phi^t_r$ the projection to $H^s(S^1,M)$ of the flow $\phi^t_r$, and by $P(t,0)$ the $L^2$-parallel transport 
along $\sigma^\cdot_r$ from $H^{1-s}((\sigma_r^0(\cdot))^*TM)$ to $H^{1-s}((\sigma_r^t(\cdot))^*TM)$.
Then,
$$\phi^t_r(\q,\p) = P(t,0)\Big [ a(t,(\q,\p)) \cdot \jmath^*_{1-s}\dot \q + b(t,(\q,\p))\cdot \p  + K(t,(\q,\p))\Big ],$$
where:
\begin{itemize}
\item $a:\R\times \M^{1-s}\to (-\infty,0]$ maps bounded sets into precompact sets and satisfies $a(0,\cdot)\equiv 0$,
\item $b:\R\times \M^{1-s}\to (0,+\infty)$ maps bounded sets into precompact sets and satisfies $b(0,\cdot)\equiv1$, and 
\item $K:\R\times \M^{1-s}\to \M^{1-s}$ is a ``compact'' fibre-preserving map such that $K(0,\cdot) \equiv 0$.
\end{itemize}
\label{replemma}
\end{lem}

\begin{rmk}
In the proposition above, by compact we mean that, for any compact set $C\subset H^s(S^1,M)$ and any bounded set $B\subset \pi^{-1}(C)$ we 
have that $K(t,B)\subset \M^{1-s}$ is precompact. 
\end{rmk}

\begin{proof}
For $t\in \R$ we denote by $\dot{\sigma^t_r}(\cdot)\in H^{s-1}( \sigma^t_r(\cdot)^*TM)$ the tangent field to  $\sigma^t_r(\cdot)\in H^s(S^1,M)$. Dropping the subscript $\dot \q$ from the covariant 
derivative and 
recalling that $\jmath^*_\ell  = (1+\nabla^*\nabla)^{-\ell}$ and 
\begin{align*}
\text{grad}\, \A_r(\q,\p) &= (\text{grad}\, \A_r(\q,\p)^\hor,\text{grad}\, \A_r(\q,\p)^\ver) \\
&= \big (\jmath_s^* \nabla^* \p - \text{grad}\, \Delta(\q,\p)^\hor, \jmath_{1-s}^* (\dot \q - \p) -\text{grad}\, \Delta(\q,\p)^\ver \big ),
\end{align*}
where $\Delta:\M^{1-s}\to \R$ is given by \eqref{cc} and $\text{grad}\, \Delta$ is computed with respect to the $\langle \cdot,\cdot \rangle_{\M^{1-s}}$-metric given by \eqref{metricms}, we compute:
\begin{align*}
D_{\frac{\diff}{\diff t} \sigma_r^\cdot} \Big (\jmath^*_{1-s} \dot{\sigma_r^\cdot}\Big) 
					&= \jmath^*_{1-s} D_{\frac{\diff}{\diff t} \sigma_r^\cdot}\dot{\sigma_r^\cdot} + [D_{\frac{\diff}{\diff t} \sigma_r^\cdot}, \jmath_{1-s}^*]\dot{\sigma_r^\cdot}\\
					&=  \jmath^*_{1-s} \nabla \Big (\frac{\diff}{\diff t} \sigma_r^\cdot\Big ) + [D_{\frac{\diff}{\diff t} \sigma_r^\cdot}, \jmath^*_{1-s}]\dot{\sigma_r^\cdot}\\
					&=\jmath^*_{1-s} \nabla \Big (\frac{\diff}{\diff t} \phi_r^\cdot\Big )^{\hor} + [D_{\frac{\diff}{\diff t} \sigma_r^\cdot}, \jmath^*_{1-s}]\dot{\sigma_r^\cdot}\\
					&= - \tilde \varphi(\phi^t_r) \cdot \jmath^*_{1-s} \nabla \Big (\jmath^*_s\nabla^* \phi^t_r\Big ) + \tilde \varphi(\phi^t_r) \cdot \jmath^*_{1-s}\nabla \text{grad}\, \Delta(\phi^t_r)^\hor+ [D_{\frac{\diff}{\diff t} \sigma_r^\cdot}, \jmath^*_{1-s}]\dot{\sigma_r^\cdot}\\
					&= - \tilde \varphi(\phi^t_r) \cdot \jmath^*_1\nabla \nabla^* \phi^t_r + \tilde \varphi(\phi^t_r) \cdot \jmath^*_{1-s}\nabla \text{grad}\, \Delta(\phi^t_r)^\hor + [D_{\frac{\diff}{\diff t} \sigma_r^\cdot}, \jmath^*_{1-s}]\dot{\sigma_r^\cdot},
\end{align*}
where 
$$\tilde \varphi(\cdot) := \frac{\varphi(\cdot)}{\sqrt{1+\|\text{grad}\, \A_r(\cdot)\|^2}}.$$
Therefore, we obtain 
\begin{align*}
D_{\frac{\diff}{\diff t} \sigma_r^t} \Big (\jmath^*_{1-s} \dot{\sigma_r^t} + \phi^t_r\Big) &= - \tilde \varphi(\phi^t_r)\cdot \jmath^*_1\nabla \nabla^* \phi^t_r + \tilde \varphi(\phi^t_r) \cdot \jmath^*_{1-s}\nabla\text{grad}\, \Delta(\phi^t_r)^\hor + [D_{\frac{\diff}{\diff t} \sigma_r^t}, \jmath^*_{1-s}]\dot{\sigma_r^t} + \Big (\frac{\diff}{\diff t} \phi_r^\cdot\Big )^{\ver}\\
							&= - \tilde \varphi(\phi^t_r)\cdot \jmath^*_1\nabla \nabla^* \phi^t_r+ \tilde \varphi(\phi^t_r) \cdot \jmath^*_{1-s}\nabla\text{grad}\, \Delta(\phi^t_r)^\hor + [D_{\frac{\diff}{\diff t} \sigma_r^t}, \jmath^*_{1-s}]\dot{\sigma_r^t} \\
							&- \tilde \varphi(\phi^t_r)\cdot \jmath_{1-s}^* \big (\dot{\sigma_r^t} - \phi^t_r\big ) +\tilde \varphi(\phi^t_r)\cdot \text{grad}\, \Delta(\phi^t_r)^\ver\\ 
							&=  - \tilde \varphi(\phi^t_r)\cdot  \big ( \jmath_{1-s}^* \dot{\sigma_r^t} +\phi^t_r \big ) + [D_{\frac{\diff}{\diff t} \sigma_r^t}, \jmath^*_{1-s}]\dot{\sigma_r^t}\\
							&+ \tilde \varphi(\phi^t_r)\cdot \big ( (1  - \jmath^*_1\nabla \nabla^* )\phi^t_r + \jmath^*_{1-s}\phi^t_r + \jmath^*_{1-s}\nabla\text{grad}\, \Delta(\phi^t_r)^\hor +\text{grad}\, \Delta(\phi^t_r)^\ver \big ) \\
							&= - \tilde \varphi(\phi^t_r)\cdot  \big ( \jmath_{1-s}^* \dot{\sigma_r^t} +\phi^t_r \big ) + \kappa_1(\phi^t_r),
\end{align*}
where
$$\kappa_1(\phi^t_r) := \tilde \varphi(\phi^t_r)\cdot \big ( (1  - \jmath^*_1\nabla \nabla^* )\phi^t_r + \jmath^*_{1-s} \phi^t_r+ \jmath^*_{1-s}\nabla\text{grad}\, \Delta(\phi^t_r)^\hor +\text{grad}\, \Delta(\phi^t_r)^\ver\big )+ [D_{\frac{\diff}{\diff t} \sigma_r^t}, \jmath^*_{1-s}]\dot{\sigma_r^t}.$$
Similarly, we see that 
$$D_{\frac{\diff}{\diff t} \sigma_r^t} \Big (\jmath^*_{1-s} \dot{\sigma_r^t} - \phi^t_r\Big) = \tilde \varphi(\phi^t_r) \cdot \big ( \jmath^*_{1-s} \dot{\sigma_r^t} - \phi^t_r\big)+\kappa_2(\phi^t_r),$$
where 
$$\kappa_2(\phi^t_r) = \tilde \varphi(\phi^t_r)\cdot \big ( (1  - \jmath^*_1\nabla \nabla^* )\phi^t_r - \jmath^*_{1-s} \phi^t_r+ \jmath^*_{1-s}\nabla\text{grad}\, \Delta(\phi^t_r)^\hor -\text{grad}\, \Delta(\phi^t_r)^\ver\big )+ [D_{\frac{\diff}{\diff t} \sigma_r^t}, \jmath^*_{1-s}]\dot{\sigma_r^t}.$$
The variation of constants formula yields now 
\begin{align}
\big (\jmath^*_{1-s} \dot{\sigma_r^t} + \phi^t_r\big )(\q,\p) &= \exp\Big (- \int_0^t \tilde \varphi(\phi^\tau_r) \diff \tau \Big )\cdot P(t,0) \Big [\jmath^*_{1-s} \dot\q + \p\Big ] \nonumber \\
						&+ \int_0^t \Big ( \exp \Big (- \int_\rho^t \tilde \varphi(\phi^\tau_r) \diff \tau\Big ) \cdot P(t,\tau)\big [ \kappa_1 (\phi^\rho_r)\big ]\, \diff \rho\nonumber \\
						&= \exp\Big (- \int_0^t \tilde \varphi(\phi^\tau_r) \diff \tau \Big )\cdot P(t,0) \Big [\jmath^*_{1-s} \dot\q + \p\Big ] + K_1 (t, (\q,\p)) \label{rep1}
						\end{align}
and on the other hand 
\begin{equation}
\big (\jmath^*_{1-s} \dot{\sigma_r^t} - \phi^t_r\big )(\q,\p) = \exp\Big ( \int_0^t \tilde \varphi(\phi^\tau_r) \diff \tau \Big )\cdot P(t,0) \Big [\jmath^*_{1-s} \dot\q - \p\Big ] + K_2(t,(\q,\p)),
\label{rep2}
\end{equation}
where 
$$K_2(t,(\q,\p)) = \int_0^t \Big ( \exp \Big (\int_\rho^t \tilde \varphi(\phi^\tau_r) \diff \tau\Big ) \cdot P(t,\tau)\big [ \kappa_2 (\phi^\rho_r)\big ]\, \diff \rho.$$
Subtracting \eqref{rep2} to \eqref{rep1} we obtain 
\begin{align*}
\phi^t_r(\q,\p)&=\underbrace{\frac 12 \Big [ \exp \Big (- \int_0^t \tilde \varphi(\phi^\tau_r) \diff \tau\Big ) - \exp\Big ( \int_0^t \tilde \varphi(\phi^\tau_r) \diff \tau \Big )\Big ]}_{:=a(t,(\q,\p))}\cdot P(t,0) \big [\jmath^*_{1-s}\dot \q\big ]\\
				&\underbrace{+\frac 12 \Big [ \exp \Big (- \int_0^t \tilde \varphi(\phi^\tau_r) \diff \tau\Big ) + \exp\Big ( \int_0^t \tilde \varphi(\phi^\tau_r) \diff \tau \Big )\Big ]}_{=:b(t,(\q,\p))}\cdot P(t,0) \big [\p\big ]\\
				&+ \frac 12 \Big (K_1(t,(\q,\p))-K_2(t,(\q,\p))\Big ).
\end{align*}
It is straightforward to check that the functions $a$ and $b$ have the desired properties. Now set
$$K(t,(\q,\p)):= \frac 12 P(0,t) \big [ K_1(t,(\q,\p))-K_2(t,(\q,\p))\big ].$$
We readily see that all the operators appearing in the functions $\kappa_1$ and $\kappa_2$ are compact, hence the fact that $K$ is a compact fibre-preserving map
follows from the fact that parallel transport ``behaves well'' with respect to compactness; for more details we refer to \cite[Section 3]{Hofer:1988}.
\end{proof}

\begin{proof}[Proof of Proposition \ref{intersectionproposition}]
In virtue of the representation Lemma \ref{replemma} we see that the problem 
$$\phi^{t}_r ( \pi^{-1}(C)) \cap \mathbb O_{H^s} \neq \emptyset, \quad t\in [0,t_0],$$ 
is equivalent to finding solutions of 
\begin{equation}
0 = a(t,(\q,\p)) \cdot \jmath^*_{1-s}\dot \q + b(t,(\q,\p))\cdot \p  + K(t,(\q,\p)),
\label{1}
\end{equation}
on $\pi^{-1}(C)$. We equivalently rewrite  \eqref{1} as 
\begin{equation}
\p = - \frac{1}{b(t,(\q,\p))} \cdot \Big (a(t,(\q,\p)) \cdot \jmath^*_{1-s}\dot \q + K(t,(\q,\p))\Big )=: T(t,(\q,\p)),
\label{2}
\end{equation}
where $T:[0,t_0]\times \pi^{-1}(C)\to \pi^{-1}(C)$ is a fibre-preserving map mapping bounded sets into precompact sets and additionally satisfying 
$$T(0,\cdot ) \equiv 0$$
and 
 $$T(t,\cdot) \equiv 0 \ \ \text{on} \ \pi^{-1}(C) \cap \{\|\p\|_{1-s}\geq \gamma''\}.$$
We are now in position to apply Dold's fixed point transfer \cite{Dold:1976} (see also \cite{Hofer:1985ur}). 
This yields a transfer homomorphism $\text{tr}_t$, $t\in [0,t_0]$, such that the following 
diagram is commutative

$$\xymatrix{ & H^*\Big (\phi^{-t}_r \Big ( \phi^{t}_r(\pi^{-1}(C))\cap \mathbb O_{H^s}\Big )\Big ) \ar[dr]_{\text{tr}_t} & \\
H^*(C) \ar[ur]^{\pi^*} \ar[rr]^{\text{id}^*} & & H^*(C)
}
$$
where with slight abuse of notation we denoted with $\pi^*$ the map induced in cohomology by 
$$\pi \Big |_{\phi^{-t}_r \big ( \phi^{t}_r(\pi^{-1}(C))\cap \mathbb O_{H^s}\big )}:  \phi^{-t}_r \big ( \phi^{t}_r(\pi^{-1}(C))\cap \mathbb O_{H^s}\big ) \to C.$$
In particular, we obtain that $\pi^*$ is injective, and hence the desired homomorphism is given by 
$$\beta_t:= (\phi^{-t}_r)^* \circ \pi^*.$$
One now easily checks the commutativity of the diagram in the statement of Proposition \ref{intersectionproposition}.
\end{proof}


\subsection{The proof.} Now we explain how Theorem \ref{thm:4} follows from Propositions  \ref{deformationproposition} and \ref{intersectionproposition}.
If $M$ is not simply-connected we choose $C=\{\gamma\}$, where $\gamma\in C^\infty(S^1,M)$ is a smooth non-contractible loop. 

If $M$ is simply connected the choice of $C$ is more subtle, since for an arbitrary $C$ we cannot exclude that the 
critical point of $\A_r$ coming from the minimax procedure be constant. We recall that Sullivan's theory of minimal models for rational homotopy type \cite{Sullivan:1973,Vigue-Poirrier:1976ug}
guarantees that the rational cohomology groups of $H^1(\T,M)$ (thus, of $H^s(\T,M)$ since they are homotopically equivalent) 
do not vanish in arbitrary large degree. Moreover, for any $k\in \N$ 
 we can find a compact set $C\subset H^1(S^1,M)$ such that the inclusion $\imath:C\hookrightarrow H^1(\T,M)$ 
 induces an isomorphism in cohomology $\imath^*:H^* (H^1(\T,M)) \to H^*(C)$ up to degree $k$ (c.f. \cite{Bott:1982}). 
 Therefore, we choose $k>\dim M$ such that $H^k(H^1(\T,M))\neq 0$ and pick $C\subset H^1(S^1,M)$ compact as above; notice that $C$ is a fortiori compact in $H^s(S^1,M)$.  
 
 In both cases, we obtain a bounded continuous non-increasing minimax function via 
 $$\theta: (0,+\infty)\to [0,+\infty), \quad \theta (r) := \inf_{t\geq 0} \sup_{\phi^t_r(\pi^{-1}(C))} \A_r.$$
 The fact that $\theta$ is non-increasing and bounded is obvious. By Proposition \ref{intersectionproposition} we also see that 
 $$\sup_{\phi^t_r(\pi^{-1}(C))} \A_r \geq \inf_{\mathbb O_{H^s}} \A_r =0,\quad \forall t\geq 0,$$
 thus $\theta(r)\geq 0$. As far as continuity is concerned, we observe that for $r_1\geq r_2$ and fixed $t\geq 0$ we have 
 (for sake of simplicity we assume that the both suprema are attained, say at $(\q_1,\p_1)$ and $(\q_2,\p_2)$ respectively)
 
 \begin{align*}
0\leq \sup_{\phi^t_{r_2}(\pi^{-1}(C))} \A_{r_2} - \sup_{\phi^t_{r_1}(\pi^{-1}(C))} \A_{r_1} &= \A_{r_2}(\q_2,\p_2) - \A_{r_1}(\q_1,\p_1)\\
								&\leq \A_{r_2}(\q_2,\p_2) - \A_{r_1}(\q_2,\p_2)\\
								&= \Delta_{r_1}(\q_2,\p_2) - \Delta_{r_2}(\q_2,\p_2 )\\
								&\leq \sup_{(q,p)\in TM} \Big (\delta_{r_1}(q,p) - \delta_{r_2}(q,p)\Big ),
\end{align*} 
where $\Delta:\M^{1-s}\to \R$ and $\delta:TM\to \R$ are as in \eqref{cc}. Therefore, we obtain (also here we assume for sake of simplicity that both 
infima are attained, say at $t_1$ and $t_2$ respectively)

\begin{align*}
0 & \leq \theta(r_2)-\theta(r_1) \\
   & =   \inf_{t\geq 0} \sup_{\phi^t_{r_2}(\pi^{-1}(C))} \A_{r_2} -  \inf_{t\geq 0} \sup_{\phi^t_{r_1}(\pi^{-1}(C))} \A_{r_1} \\
   &= \sup_{\phi^{t_2}_{r_2}(\pi^{-1}(C))} \A_{r_2} - \sup_{\phi^{t_1}_{r_1}(\pi^{-1}(C))} \A_{r_1}\\
   &\leq   \sup_{\phi^{t_1}_{r_2}(\pi^{-1}(C))} \A_{r_2} - \sup_{\phi^{t_1}_{r_1}(\pi^{-1}(C))} \A_{r_1}\\
   &\leq \sup_{(q,p)\in TM} \Big (\delta_{r_1}(q,p) - \delta_{r_2}(q,p)\Big ),
   \end{align*}
and the claim follows. Theorem \ref{thm:4} finally follows from the next

\begin{lem}
For every $r>0$ there exists $(\q_r,\p_r)\in\mathrm{crit}\, \A_r$ non-constant with $\A_r(\q_r,\p_r)=\theta(r)$.
\label{lem:existencensc}
\end{lem}
\begin{proof}
The fact that $\theta(r)$ is a critical value for $\A_r$ follows from Proposition \ref{deformationproposition}. In case $M$ is not simply-connected, the fact that 
the corresponding critical point $(\q_r,\p_r)$ is non-constant follows from the fact that we are working on a connected component of non-contractible loops. 

In case $M$ is simply connected we need a more refined argument to exclude that $(\q_r,\p_r)$ be constant; this will make use of the assumptions on 
the compact set $C$. We first notice that $(\q_r,\p_r)$ is necessarily non-constant if $\theta(r)>0$, as constant critical points have non-positive $\A_r$-action. 
Therefore, we can assume that $\theta(r)=0$ and that all critical points of $\A_r$ at level zero are constant. 

We start noticing that a sufficiently small neighborhood $\mathcal U\subset H^s(S^1,M)$ of the set $\Lambda^0M$ of constant loops (which we recall is diffeomorphic to $M$) 
cannot contain non-constant closed geodesics for $(M,g)$. This follows from the fact that, since $s>\frac 12$, $H^s$-closedness to a constant loop implies 
$C^0$-closedness, and the claim follows from the positivity of the injectivity radius of $(M,g)$. In particular, the image of any loop in $\U$ is contained in a 
small Riemannian ball. From this we see that $\Lambda^0M$ is a strong deformation retract 
of $\mathcal U$: Indeed, we first ``regularize'' loops in $\mathcal U$ to obtain a set $\{\E<\epsilon\}\subset H^1(S^1,M)$, $\epsilon>0$ small enough, 
and then use the negative gradient flow of the energy functional $\E$, as recalled in the proof of Lemma \ref{lembumpy}, to deform $\{\E<\epsilon\}$ into $\Lambda^0M$. 

By assumption we now have that $\V:= \pi^{-1}(\U)$ is a neighborhood of 
 $$\crit (\A_r) \cap \A_r^{-1}(0).$$
Thus, Proposition \ref{deformationproposition} yields $\epsilon>0$ and $t_0>0$ such that 
 for all $t\geq t_0$
$$\phi^t_r \big (\{\A_r\leq \epsilon\}\setminus \V\big ) \subset \{ \A_r \leq -\epsilon\}.$$
Using the definition of $\theta(r)$, we find $t_1\geq 0$ such that
 $$\phi^{t_1}_r(\pi^{-1}(C))\subset \{\A_r \leq \epsilon\}.$$
 Therefore,
 $$\phi^{t_0}_r \Big (\phi^{t_1}_r (\pi^{-1}(C) )\setminus \V\Big )\subset \{ \A_r\leq -\epsilon\},$$
 which implies that 
 $$\phi^{t_0}_r \Big (\phi^{t_1}_r (\pi^{-1}(C))\setminus \V\Big )\cap \mathbb O_{H^s} =\emptyset.$$
 Since $\phi^{t_0+t_1}_r (\pi^{-1}(C)) \cap \mathbb O_{H^s}\neq \emptyset$ by Proposition \ref{intersectionproposition}, we deduce that 
 $$\phi^{t_0+t_1}_r (\pi^{-1}(C)) \cap \mathbb O_{H^s}\subset \phi^{t_0}_r(\V).$$
 Using again Proposition \ref{intersectionproposition} we obtain that the diagram
 $$
 \xymatrix{
H^* (\phi^{t_0}_r(\V)) \ar[rrr]^{\jmath^*} & & & H^* (\phi^{t_0+t_1}_r(\pi^{-1}(C)) \cap \mathbb O_{H^s})\\
  & H^*(H^s(S^1,M)) \ar[ul]^{(\pi|_{\phi^{t_0}_r(\V)})^*\ \ \ } \ar[urr]_{\ (\pi|_{...})^*} \ar[rr]_{\imath^*}& & H^*(C) \ar[u]_{\beta_{t_0+t_1}}
 }
 $$
 commutes. Thus, the fact that $\beta_{t_0+t_1}$ is injective implies that the map
  $$\jmath^*_k \circ (\pi|_{\phi^{t_0}_r(\V)})^*_k :  H^k(H^s(S^1,M)) \to  H^k (\phi^{t_0+t_1}_r(\pi^{-1}(C)) \cap \mathbb O_{H^s})$$
 is non-zero and injective, and this contradicts the fact that 
\begin{equation*}
H^k(\phi^{t_0}_r(\V))\cong  H^k (\V) \cong  H^k (\U) \cong H^k(M) =0. \qedhere
\end{equation*}
\end{proof}


\appendix

\section{Proof of Lemma \ref{lem:2.2}}
\label{app:A}

In this section we give a proof of Lemma \ref{lem:2.2} on the growth speed of the eigenvalues of the self-adjoint operator $\nabla_{\dot \q}^* \circ \nabla_{\dot \q}$, 
for a given smooth loop $\q\in C^\infty(S^1,M)$. Moreover, we provide a uniform bound for the $L^\infty$-norm of the corresponding eigenvectors with $L^2$-norm equal one.

We consider a time-depending local chart $\varphi:S^1 \times B_\epsilon(0)\to M$ with $\varphi(\cdot,0)=\q$ and the induced map 
$$C^\infty (S^1,\R^n) \to \Gamma(\q^*TM), \quad \xi \mapsto (t\mapsto \diff \varphi(t,0) \cdot \xi (t)).$$
In this setting we have 
\begin{equation}
\nabla_{\dot \q} \xi = \dot \xi + \Gamma(\cdot,\dot \q(\cdot))\cdot \xi,
\label{covdev}
\end{equation}
with 
\begin{equation}
 |  \Gamma(\cdot,\dot \q(\cdot))| \leq \alpha \|\dot \q\|_\infty
 \label{christoffelsymbols}
 \end{equation}
for some constant $\alpha>0$ depending only on $g$. The quadratic form $Q:C^\infty(S^1,\R^n)\to \R$ associated with the self-adjoint operator $\nabla_{\dot \q}^* \circ \nabla_{\dot \q}$ reads 
$$Q(\xi) := \int_0^1 | \nabla_{\dot \q}\xi |^2 \, \diff t.$$
Using \eqref{covdev}, \eqref{christoffelsymbols}, and the elementary 
inequality $(a+b)^2\leq 2 (a^2+b^2)$, we compute 
\begin{align*}
Q(\xi) &= \int_0^1 | \dot \xi + \Gamma(\cdot,\q(\cdot))\cdot \xi|^2 \, \diff t\\
	&\leq 2 \int_0^1 \Big (|\dot \xi|^2 + |\Gamma(\cdot,\q(\cdot))\cdot \xi|^2\Big ) \, \diff t\\
	&\leq 2 \int_0^1 \Big (|\dot \xi|^2 + \alpha^2 \|\dot \q\|_\infty |\xi|^2\Big ) \, \diff t\\ 
	&\leq D \int_0^1 \Big (|\dot \xi|^2_{\text{eucl}} + E(\|\dot \q\|_\infty) |\xi|^2_{\text{eucl}}\Big )\, \diff t =: Q^+(\xi),
\end{align*}
where $D,E(\|\dot \q\|_\infty)>0$ are suitable constants depending respectively only on the metric $g$ and on the metric and the $L^\infty$-norm of $\dot \q$.
Similarly, employing the inequality $(a-b)^2 \geq \frac 12 a^2 - b^2$ we obtain 
\begin{align*}
Q(\xi) &= \int_0^1 | \dot \xi + \Gamma(\cdot,\q(\cdot))\cdot \xi|^2 \, \diff t\\
	&\geq \int_0^1 \Big (\frac 12 |\dot \xi|^2 - |\Gamma(\cdot,\q(\cdot))\cdot \xi|^2\Big ) \, \diff t\\
	&\geq d \int_0^1 \Big (|\dot \xi|^2_{\text{eucl}} -e(\|\dot \q\|_\infty) |\xi|^2_{\text{eucl}}\Big )\, \diff t =:Q^-(\xi),
\end{align*}
where again $d,e(\|\dot \q\|_\infty)>0$ are suitable constants.  From the variational characterization of the eigenvalues of a self-adjoint operator $T$ 
on a Hilbert space $\HH$
$$\lambda_j(T) = \max_{\text{codim} (V) =j} \min_{S\cap V} \ Q,$$
where $Q$ is the associated quadratic form and $S\subset \HH$ is the unit sphere, we deduce that 
$$\lambda_j (Q^-) \leq \lambda_j (\q) \leq \lambda_j (Q^+),$$
and it is now an easy exercise to show that 
$$\left \{\begin{array}{l} \lambda_j(Q^-) = c (j^2 - d(\|\dot \q\|_\infty)), \\ \\ \lambda_j(Q^+) = C (j^2 + d(\|\dot \q\|_\infty)). \end{array}\right., \quad  \quad \forall j.$$
Indeed, the operator associated with $Q^-$ (the argument being analogous for $Q^+$) is given by 
$$\xi \mapsto - d \big ( \ddot \xi + e(\|\dot \q\|_\infty) \xi\big ),$$
and hence its eigenvalues are given by $d(4\pi^2j^2 - e(\|\dot \q\|_\infty)$. 

Let now $\xi$ be an eigenvector of $\nabla_{\dot \q}^* \circ \nabla_{\dot \q}$ with $\|\xi\|_2 =1$, and let $\lambda^2>0$ be the corresponding eigenvalue, 
that is $-\nabla_{\dot \q}^2\xi = \nabla_{\dot \q}^* \circ \nabla_{\dot \q} \xi = \lambda^2 \xi$. We set 
$$u:= (\xi, \frac 1 \lambda \nabla_{\dot \q} \xi ) \in \Gamma(\q^*TM)\times \Gamma(\q^*TM),$$
where $\Gamma(\q^*TM)\times \Gamma(\q^*TM)$ is endowed with the product $L^2$-metric, and compute 
\begin{align*}
|u(t_1)|^2 - |u(t_0)|^2 &= \int_{t_0}^{t_1} \dt |u(t)|^2\, \diff t\\
				&= 2\int_{t_0}^{t_1} g_\q (\nabla_{\dot \q} u,u)\, \diff t\\
				&= 2 \int_{t_0}^{t_1} \Big (g_\q (\nabla_{\dot \q} \xi, \xi) + g_\q (\frac 1 \lambda \nabla_{\dot \q}^2 \xi, \frac 1\lambda \nabla_{\dot \q} \xi )\Big )\, \diff t\\
				&= 0. 
\end{align*}
It follows that the function $t\mapsto |u(t)|$ is constant. In particular, 
$$c = \|u\|^2 = \|\xi\|^2 + \| \frac 1\lambda \nabla_{\dot \q} \xi \|^2 = 1 + \int_0^1 \frac{1}{\lambda^2} g_\q (\nabla_{\dot \q}^* \circ \nabla_{\dot \q} \xi ,\xi) \, \diff t = 2,$$
so that $|\xi(t)|^2 \leq |u(t)|^2 \leq 2$ for all $t\in [0,1]$, an the claim follows.


\section{Non global equivalence of the metrics $\langle \cdot,\cdot\rangle_r$ and $\langle \cdot,\cdot\rangle_1r\emb$}
\label{app:B}

In this section we provide an example showing that the metrics $\langle \cdot,\cdot\rangle_r$ and $\langle \cdot,\cdot\rangle_r^\emb$ defined in Section 2 are not globally 
equivalent for every $r\in (0,1]$ (notice that for $r=0$ the two metrics coincide by construction). 

Thus, let 
$$M := S^1 = \{z \in \mathbb{C}: |z|^2 = 1\} \subset \mathbb{C} \simeq \mathbb{R}^2$$
be the unit circle endowed with the restriction of the euclidean metric. Set
$$\q_n(t) := e^{2\pi int},\quad \p_n(t) := ie^{2\pi int}, \quad \forall n\in\N.$$ 

For fixed $n\in\N$, we observe that, for every $t\in \T$, the vectors $\q_n(t)$ and $\p_n(t)$ form an orthonormal basis of $T_{\q_n(t)}\mathbb{R}^2$, and $\p_n(t)  \in T_{\q_n(t)}S^1$.
In particular, $\p_n \in T_{\q_n} H^s(S^1,M)$. For any $\mathbf{w}\in\Gamma(\q_n^*TS^1)$ we have 
\[
\dot{\mathbf w}(t) = \langle \dot{\mathbf w}(t),\p_n(t)\rangle\cdot \p_n(t) + \langle \dot{\mathbf w}(t),\q_n(t)\rangle \cdot \q_n(t) = \nabla_{\dot{\q}_n} {\mathbf w}(t) + \langle \dot{\mathbf w}(t),\q_n(t)\rangle \cdot \q_n(t).
\]
Differentiating the identity $\langle {\mathbf w}(t),\q_n(t)\rangle = 0$ we get 
\[
\langle \dot{\mathbf w}(t),\q_n(t)\rangle = - \langle {\mathbf w}(t),\dot{\q}_n(t)\rangle.
\]
We can now estimate
\begin{align*}
\|\w\|_{1}^2 \leq (\|\w\|_{1}^{\emb})^2 &= \|\w\|^2 + \|\dot{\w}\|^2 = \|\w\|^2 + \|\nabla_{\dot \q_n}\w\|^2 + \| \langle \w(t),\dot{\q}_n(t) \rangle \cdot  \q_n(t)\|^2 \\
& \leq \|\w\|_1^2 + \|\w\|^2 \cdot \|\dot{\q}_n\|^2 \leq (1+(2\pi n)^2)\|\w\|_1^2,
\end{align*}
that is, $\| \cdot\|_1$ and $\|\cdot\|_{1}^\emb$ are equivalent on $\Gamma(u_n^*TS^1)$. By the L\"owner-Heinz theorem, the norms $\|\cdot\|_r$ and 
$\|\cdot\|_r^\emb$ are equivalent on $\Gamma(\q_n^*TS^1)$ for every $r\in [0,1]$. 

On the other hand, we readily see that, for $r\in (0,1]$, there is no constant $c$ independent of $n$ such that 
$\|\cdot\|_{r}^\emb \leq c\|\cdot\|_r$. Indeed, for $\w = \p_n$ we have
$$(1+ \nabla_{\dot \q_n}^*\circ \nabla_{\dot \q_n} )\p_n = \p_n, \quad (1-\Delta)\p_n = \big (1+(2\pi n)^2\big )\p_n,$$
where we used the fact that 
$$\nabla_{\dot \q_n} \p_n (t) = \text{pr}_{T_{\q_n(t)}S^1} \dot \p_n(t) = \text{pr}_{T_{\q_n(t)}S^1} \Big (-(2\pi n)^2 \q_n(t)\Big ) =0,\ \ \forall t\in \T.$$
Therefore,
$$\|\p_n\|_r =\|\p_n\| \equiv 1,\quad \|\p_n\|_{r}^\emb = \big (1+(2\pi n)^2\big )^r \to \infty \ \ \text{as}\ \ n\to +\infty.$$ 
In particular, the two norms are not globally equivalent.

\bibliography{_biblio}
\bibliographystyle{plain}

\end{document}